\documentclass[12pt,leqno]{article}

\usepackage[utf8]{inputenc}
\usepackage{amsmath,amssymb,epsfig,bbm}
\usepackage{amsthm}
\usepackage{stmaryrd,mathabx}
\usepackage{comment}
\usepackage{color}
\usepackage[T1]{fontenc}

\usepackage{lmodern}
\usepackage{tikz}
\usepackage{adjustbox}
\usepackage{float}

\usepackage{interval}

\usepackage[textsize=small]{todonotes}
\usepackage{enumitem}
\usepackage{varwidth}
\usepackage{csquotes}
\usepackage{mathrsfs}

\usepackage{mathtools}

\DeclarePairedDelimiter\floor{\lfloor}{\rfloor}

\newtheorem{defi}{Definition}[section]
\newtheorem{prop}[defi]{Proposition}
\newtheorem{theorem}[defi]{Theorem}

\newtheorem{lemma}[defi]{Lemma}

\newtheorem{corollary}[defi]{Corollary}
\newtheorem{remarks}[defi]{Remarks}
\newtheorem{remark}[defi]{Remark}

\newcommand{\bdefi}{\begin{defi}}
\newcommand{\edefi}{\end{defi}}
\newcommand{\bprop}{\begin{prop}}
\newcommand{\eprop}{\end{prop}}
\newcommand{\btheo}{\begin{theo}}
\newcommand{\etheo}{\end{theo}}
\newcommand{\btheofr}{\begin{theofr}}
\newcommand{\etheofr}{\end{theofr}}
\newcommand{\blemm}{\begin{lemm}}
\newcommand{\elemm}{\end{lemm}}
\newcommand{\blemmfr}{\begin{lemmfr}}
\newcommand{\elemmfr}{\end{lemmfr}}
\newcommand{\brema}{\begin{rema}}
\newcommand{\erema}{\end{rema}}
\newcommand{\bexer}{\begin{exem}}
\newcommand{\eexer}{\end{exem}}
\newcommand{\bexems}{\begin{exems}}
\newcommand{\eexems}{\end{exems}}
\newcommand{\bconj}{\begin{conj}}
\newcommand{\econj}{\end{conj}}
\newcommand{\bcoro}{\begin{coro}}
\newcommand{\ecoro}{\end{coro}}

\usepackage{mathrsfs}
\renewcommand\mathcal{\mathscr}

\newcommand{\F}{{\cal F}}

\newcommand{\OOO}{{\cal O}}
\renewcommand{\P}{{\cal P}}

\newcommand{\R}{{\cal R}}

\newcommand{\maths}[1]{{\mathbb #1}}  

\newcommand{\CC}{\maths{C}}

\newcommand{\NN}{\maths{N}}

\newcommand{\QQ}{\maths{Q}}
\newcommand{\RR}{\maths{R}}
\newcommand{\SSS}{\maths{S}}

\newcommand{\ZZ}{\maths{Z}}
\def\11{{\mathbbm 1}}

\newcommand{\weakstar}{\overset{*}\rightharpoonup}

\def\e{\varepsilon}

\newcommand{\bigO}{\operatorname{O}}

\newcommand{\card}{{\operatorname{Card}}}

\newcommand{\covol}{\operatorname{covol}}

\newcommand{\diam}{{\operatorname{diam}}}

\renewcommand{\Im}{{\operatorname{Im}}}

\newcommand{\Leb}{\operatorname{Leb}}

\renewcommand{\Re}{{\operatorname{Re}}}

\newcommand{\vol}{\operatorname{vol}}

\newcommand{\supp}{\operatorname{supp}}
\newcommand{\lvl}{\operatorname{lvl}}

\newcommand{\sys}{\operatorname{sys}}
\newcommand{\Err}{\operatorname{Err}}
\newcommand{\Grid}{\operatorname{Grid}}
\newcommand{\Lat}{\operatorname{Lat}}
\newcommand{\Pow}{\operatorname{Pow}}

\usepackage[
backend=biber,
style=alphabetic,
sorting=debug,
]{biblatex}

\renewbibmacro{in:}{%
	\ifentrytype{article}
	{}
	{\bibstring{in}%
		\printunit{\intitlepunct}}}

\DeclareFieldFormat{pages}{#1} 

\newcommand\numberthis{\addtocounter{equation}{1}\tag{\theequation}}

\usepackage{orcidlink}

\usepackage{hyperref}
\hypersetup{
	pdfencoding=auto,
	colorlinks=true,
	citecolor=orange,
	filecolor=black,
	linkcolor=blue,
	urlcolor=blue,
	bookmarksnumbered=true}

\begin{document}
\thispagestyle{plain}
\begin{center}
	\Large
	\textbf{Effective statistics of pairs of\\fractional powers of complex grid points}
	
	\large
	\vspace{0.3cm}
	\hspace{0.35cm} Rafael Sayous \orcidlink{0009-0007-6306-8546}
	
	\normalsize
	\vspace{0.1cm}
	\today
\end{center}
\noindent \textbf{Abstract}: Using a standard definition of fractional powers on the universal cover $\exp:S\to \CC^*$ where $S$ is the standard infinite helicoid embedded in $\RR^3$, we study the statistics of pairs at various scalings from the countable family $\{n^\alpha \, : \, n \in \exp^{-1}(\Lambda) \}$ for every complex grid $\Lambda$ and every real parameter $\alpha \in \, ]0,1[\,$. We prove the convergence of the empirical pair correlation measures towards a rotation invariant measure with explicit density. In particular, with the scaling factor $N\mapsto N^{1-\alpha}$, we prove that there exists an exotic pair correlation function which exhibits a level repulsion phenomenon. For other scaling factors, we prove that either the pair correlations are Poissonian or there is a total loss of mass. We give an error term for this convergence.



\vspace{-0.25cm}
\section{Introduction}\label{sec:intro}
Let $G$ be a locally compact metric additive group. In order to comprehensively understand the distribution of a countable family $(u_i)_{i\in I}$ in $G$, an essential aspect involves analysing the statistics of the spacings between selected pairs of these points, seen at various scalings. The approach consisting in taking all pairs into account is the study of \emph{pair correlations}. More precisely, let $\phi : [0,+\infty[ \, \to G^G$ be a scaling function, and $h: I \to [0,+\infty]$ be a height function (i.e.~a nonnegative function that every set $\{i \in I \, : \, h(i) \leq N\}$ is finite). Our focus lies on the asymptotic of the multisets $F_N=\{ \phi(N)(u_i-u_j) \}_{h(i),h(j) \leq N, i \neq j}$ as $N \to \infty$.

These problems initially occurred in physics, especially in quantum chaos, which has lead to a purely mathematical point of view of pair correlations. See for instance \cite{rudnick1998paircorrel_fracpartpoly, aichinger2018quasienergyspectra_pairs, larcher2020pair_maximaladdenergy} for questions directly linked to quantum physics. Determining the behaviour of pair correlations for a deterministic numerical sequence may present an intriguing challenge, see the papers \cite{rudnick1998paircorrel_fracpartpoly, boca2005pairs_farey, larcher2018poissonpairs_negativeresults, lutsko2022poissonianslowlygrowingseq, lutsko2023longrangepaircor, paulinparkko2022a_pairs_log, paulinparkko2024_pairs_complexlog}. For instance, when $\alpha >0$ is small enough, the sequence $(\{n^\alpha\})_{n\in\NN}$, where $\{\cdot\}$ denotes the fractional part function, exhibits a behaviour commonly called \emph{Poisson pair correlations}, as proven by C.~Lutsko, A.~Sourmelidis and N.~Technau in their paper \cite{lutsko2021pairfracpowers}, as well as in the special case $\alpha=\frac{1}{2}$, as shown by D.~El-Baz, J.~Marklof and I.~Vinogradov in \cite{elbaz2015two_point_correl_fract_sqrtn_poisson}.

In our setting, the metric group $G$ will then be $(\CC,+)$. Recall that a complex \emph{$\ZZ$-lattice} is a discrete additive subgroup of $\CC$ generating $\CC$ as a real vector space, and that a complex subset $\Lambda$ is called a \emph{$\ZZ$-grid} if there exist a (unique) $\ZZ$-lattice $\vec{\Lambda}$ and a complex number $z \in\CC$ such that $\Lambda=z+\vec{\Lambda}$. The spaces $\Lat_\CC$ of complex $\ZZ$-lattices, and $\Grid_\CC$ of complex $\ZZ$-grids, are endowed with the Chabauty topology (since lattices and grids are closed subsets of $\CC$). In this introduction, all grids and lattices are assumed to be unimodular (i.e.~of covolume $1$ such as the lattice $\ZZ[i]$). In what follows, we fix a real number $\alpha \in \, ]0,1[\,$ and a unimodular $\ZZ$-grid $\Lambda \in \Grid_\CC$. We have chosen to widen our focus, working with grids instead of lattices only, since grids have become trendy in number theoretical issues, see for instance \cite{elkiesmcmullen2004gapssqrtn,shapira2013gridsdensevalues, akaeinsiedlershapira2016integerptssphereorthogrids, limsaxceshapira2019dimbadgrid, moshraoshapira2024badgridskdivlat}. Let $\gamma \in \, ]0,1[\,$, that we use as a parameter for the scaling in this introduction. To conduct a much more involved study than the paper \cite{sayous2023realpaircor} on the pair correlation statistics of the real sequence $(n^\alpha)_{n\in\NN}$, we will define a sequence of measures for the pair correlations of the "$\alpha$ powers" of grid points in $\Lambda$. In this introduction, we present the case $\alpha = \frac{1}{b}$ where $b\in\NN-\{0\}$. In this particular case, the study we conduct can be simplified and translated to the statistics of scaled differences $N^\gamma (v-u)$ where $u$, $v$ are $b$-th roots of grid points with norm less than $N$. Such a scaling factor $N^\gamma$ is a usual choice, see \cite{weiss2023explicit, nairpollicott2007paircorhighdim}. In other words, we study the sequence of \emph{empirical pair correlation measures} given by
$$ \R_N = \frac{\alpha}{N^{2(2-\alpha-\gamma)}} \sum_{\substack{n, m \in \Lambda, \, n \neq m \\ 0 < |n|,|m|\leq N}} \; \sum_{\substack{u,v \in \CC^* \\ u^b = m, \, v^b=n}} \Delta_{N^\gamma(v-u)},$$
where, for all complex number $z\in\CC$, we denote by $\Delta_z$ the Dirac mass at $z$. We denote by $\Leb_\CC$ the Lebesgue measure on $\CC$, and we define the nonnegative measurable function $\rho=\rho_{\alpha,\gamma,\vec{\Lambda}}$ by
$$
\rho : z \mapsto \left\{
\begin{array}{clc}
	\displaystyle 0 & \mbox{ if } & \gamma > 1-\alpha,\vspace{2mm}\\
	\displaystyle \frac{\pi}{\alpha^2(2-\alpha)} & \mbox{ if } & \gamma<1-\alpha, \vspace{2mm}\\
	\displaystyle \frac{\alpha^\frac{2}{1-\alpha}}{(1-\alpha)} |z|^{-\frac{4-2\alpha}{1-\alpha}}  \sum_{\substack{p\in\vec{\Lambda} \\ |p| \leq \frac{|z|}{\alpha}}} |p|^{\frac{2}{1-\alpha}} & \mbox{ if } & \gamma=1-\alpha.
\end{array} \right.
$$
We use the notation $D(z_0,r)=\{z\in\CC \, : \, |z-z_0| < r\}$ for open disks. For all Radon measures $\mu_N$, for $N\in\NN$, and $\mu$ on $\CC$, the sequence $(\mu_N)_{N\in\NN}$ is said to \emph{vaguely converges} towards $\mu$ if for every continuous functions $f:\CC \to \CC$ with compact support, we have the convergence $\mu_N(f) \to \mu(f)$. In this case, we write $\mu_N \weakstar \mu$.
\begin{theorem}\label{th:intro}
	We have the following vague convergence, as $N \to\infty$,
	$$
	\R_N \weakstar \rho \, \Leb_\CC.
	$$
\end{theorem}
This result will be proven effective in the following sense: let $f \in C_c^1(\CC)$, choose $A>1$ such that $\supp f \subset D(0,A)$ and assume that $\gamma=1-\alpha$. Then, we have a rate for this convergence, given by the estimate, as $N\to\infty$,
$$
\R_N(f) = \int_{\CC}f(z)\rho(z) \, dz + \bigO_{\alpha,\Lambda} \Big(\frac{A^4(\|f\|_\infty + \|df\|_\infty)}{N}\Big).
$$
Theorem \ref{th:intro} indicates that $\rho$ describes the pair correlations of $\alpha=\frac{1}{b}$ powers of grid points. This is essentially a particular case of Theorem \ref{th:main}, the main result of the present paper which holds for every real number $\alpha \in \, ]0,1[\,$ and for which we give an error term in Remark \ref{rk:error_term_main}. The proof of Theorem \ref{th:intro} using Theorem \ref{th:main} and some counting lemma is done at the very end of Section \ref{sec:remove_branch_cut}.
\begin{figure}[ht]
	\centering
	\scalebox{0.75}{
		\begin{adjustbox}{clip, trim=3.2cm 0.8cm 3.5cm 3.cm, max width=\textwidth}
			\includegraphics{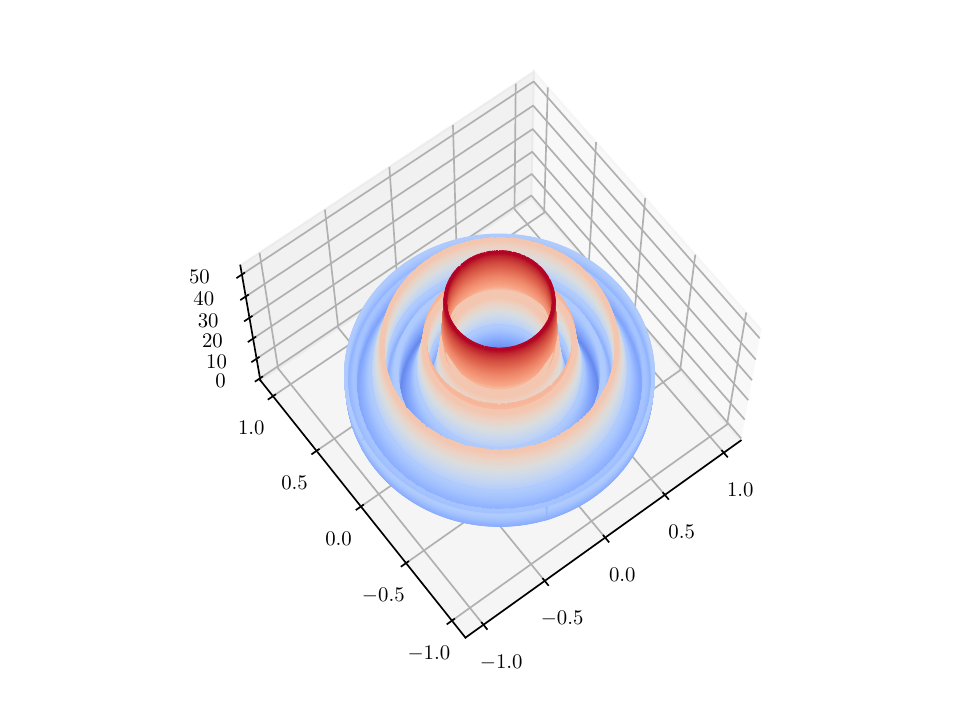}
		\end{adjustbox}
	}
	\scalebox{0.75}{
		\begin{adjustbox}{clip, trim=3.2cm 1.5cm 3.5cm 3.cm, max width=\textwidth}
			\includegraphics{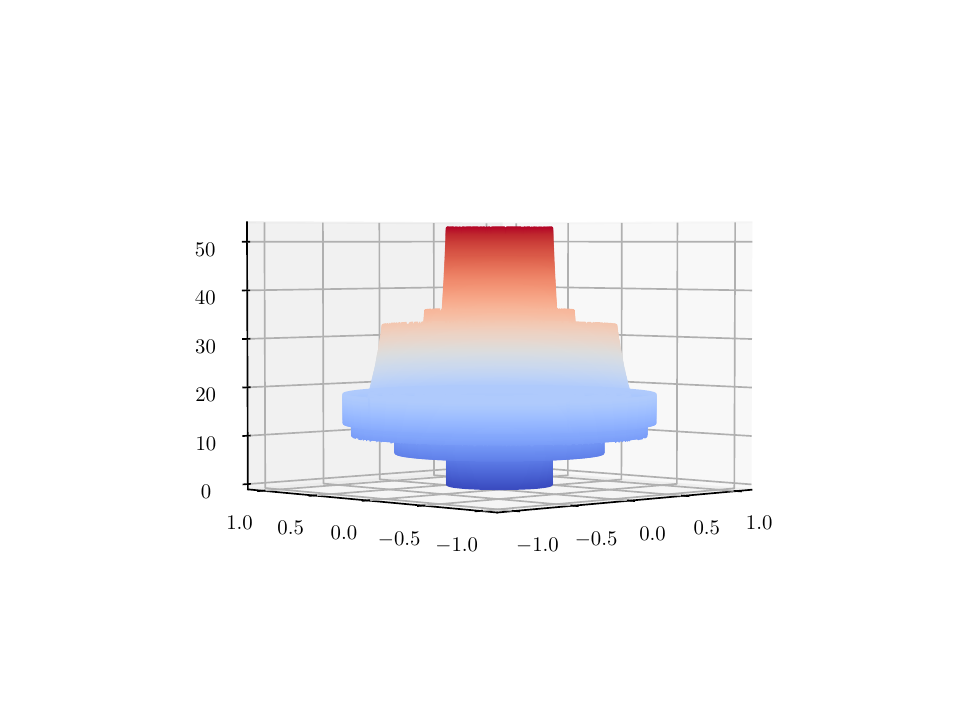}
		\end{adjustbox}
	}
	\caption{The graph of the function $\rho$ restricted to the disk $D(0,1)$ in the case $\alpha=\frac{1}{3}$, $\gamma=1-\alpha=\frac{2}{3}$ and $\Lambda=\ZZ[i]$.}
	\label{fig:th_density_alpha13}
\end{figure}

In other words, the pair correlations for the $b$-th roots of grid points have a constant density if $\gamma<1-\alpha$ (we say that these pair correlations exhibit a \emph{Poisson behaviour}), have an exotic density if $\gamma=1-\alpha$ and there is a total loss of mass if $\gamma>1-\alpha$. This phase transition phenomenon frequently appears in the study of pair correlations, see for instance \cite{paulinparkko2022a_pairs_log, paulinparkko2024_pairs_complexlog, sayous2023realpaircor}. We must insist here that we are not looking for any pseudorandom behaviour: the set $\Lambda$ is a typical example of a well distributed set (when seen from afar), and we are interested in the way the function $z \mapsto z^\alpha$ (which is transcendental if $\alpha \notin \QQ$) modifies this set at the level of pair correlations, since of course this function does not preserve gaps.

The study of pair correlation in a noncompact setting has already been fruitful in various fields. On $G=\RR$, the lengths of closed geodesics in negative curvature have Poisson pair correlation or converge to an exponential probability measure (depending on the scaling factor) \cite{pollicott1998correlpairsclosedgeod, paulinparkko2022c_exponcounting_paircorrel}. Still on $G=\RR$, for real points $\alpha, \beta$ verifying some diophantine condition, the image of $\ZZ^2$ by the quadratic form $(x,y) \mapsto (x-\alpha)^2 + (y-\beta)^2$ also exhibits a Poisson pair correlation \cite{marklof2003pair_cor_quad_form} (see also \cite{marklof2002quadformII} for a related result in higher dimension). On the group $G=(K,+)$ where $K$ is a $p$-adic field with integer ring denoted by $\OOO$, the pair correlations of squares of integers $\{ z^2 \, : \, z \in \OOO \}$ has also been studied in \cite{zaharescu2003paircors_square_padic} and has a behaviour which can arguably be called Poisson.

In Section \ref{sec:lemmas}, we first define a more general setting for pair correlations than the one of Theorem \ref{th:intro}, using the universal cover $\CC$ of $\CC^*$ and dividing it into levels: this novel technical step which will allow us to retrieve some algebraic properties of integer powers for fractional ones, giving us technically handy geometric interpretations of the studied pairs throughout the paper. Then, we state Theorem \ref{th:main}, which is the main theorem in this paper and of which Theorem \ref{th:intro} is a special case, as well as a version using separated levels, namely Theorem \ref{th:effective_cv_complex_correlations}, and we end this section by proving the main lemmas we will use for the proof of the latter theorem. In Section \ref{sec:proof_thm}, we prove Theorem \ref{th:effective_cv_complex_correlations}, using a linear approximation, an approximation of Riemann sums after appropriate changes of variable defined locally (depending on the levels introduced in Section \ref{sec:lemmas}), an averaging argument over levels (which is necessary to avoid discrepancy as illustrated in Figure \ref{fig:correlpairs_nonrotinv_exotic_alpha2342}), and various counting results. In Section \ref{sec:remove_branch_cut}, we give an upper bound on the number of pairs which were counted out by separating the grid points into levels in Section \ref{sec:lemmas}, allowing us to straightforwardly derive Theorems \ref{th:main} from Theorem \ref{th:effective_cv_complex_correlations}. The change of variable step is inspired by the unfolding technique, illustrated in \cite[§~2.1]{marklof2002quadformII}. But this paper cannot be reduced to the unfolding technique, in particular for obtaining the error terms.

\begin{figure}[ht]
	\centering
	\scalebox{0.98}{
		\begin{adjustbox}{clip, trim=2.5cm 0.85cm 2.5cm 1.45cm, max width=\textwidth}
			\includegraphics{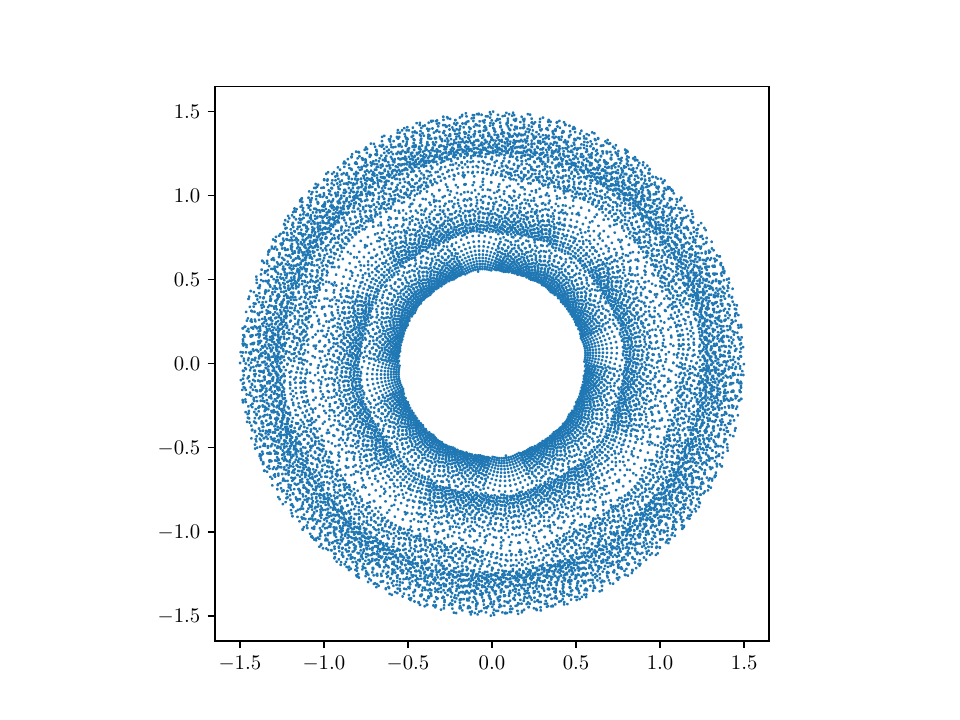}
		\end{adjustbox}
	}
	\caption{The complex points $N^{\frac{19}{42}}(n^\frac{23}{42}-m^\frac{23}{42})$ of only "one level" (with the notations of Section \ref{ssec:levels}, these are the points $N^{\frac{19}{42}}(n^{[\frac{23}{42},0]}-m^{[\frac{23}{42},0]})$) inside the disk $D(0,\frac{3}{2})$, for lattices points $m,n \in \ZZ[i]$ with $0 < |m|,|n| \leq N = 20$.}
	\label{fig:correlpairs_nonrotinv_exotic_alpha2342}
\end{figure}

\vspace{-0.03cm}One may consider to generalise Theorem \ref{th:intro} to any discrete set of constant density instead of a complex grid, and we expect the error term given after Theorem \ref{th:intro} (or the more precise version given in Remark \ref{rk:error_term_main} for Theorem \ref{th:main}) to be particularly more complicated to compute.

\section{The main statement and technical lemmas}\label{sec:lemmas}
In all this paper, we fix $\alpha \in \, ]0,1[\,$ as well as $\Lambda$ a $\ZZ$-grid in $\CC$ (not necessarily unimodular). We denote by $\vec{\Lambda}$ its underlying $\ZZ$-lattice. We set
$$S = \{ (r e^{i\omega}, \omega) \, : \, r>0, \omega \in \RR \} \subset \CC \times \RR$$
A standard way in complex analysis to define a power function is to use the Riemann surface $S$. On the universal cover $\exp:\CC\to\CC^*$ of $\CC^*$, we set $\Lambda_S = \exp^{-1}(\Lambda)$, which consists of infinitely many copies of the grid $\Lambda$ (minus the origin if $\Lambda$ contains $0$): for every $t \in \RR$, the map $\exp$ restricts to a bijection $\Lambda_S \cap \{z \, : \, t \leq \Im(z) < t+2\pi \} \to \Lambda$. We use the identification $z \mapsto (\exp(z), \Im(z))$ between the universal cover $\CC$ and the helicoid $S$. The set $\Lambda_S$ is then identified with $\{ (n,\omega) \, : \, n\in\Lambda, \; \omega \in \arg(n)\} \subset S$. We define the $\alpha$ power function on this surface by
$$
\begin{array}{rrcl}
	\Pow_\alpha : & S & \to & S
	\\ & (re^{i\omega}, \omega) & \mapsto & (r^\alpha e^{i\alpha\omega}, \alpha \omega),
\end{array}
$$
which corresponds to the multiplication by $\alpha$ on the universal cover $\CC$. We are then interested in pair correlations of the countable set $\Pow_\alpha(\Lambda_S)$. Let $\pi_\CC$ (resp.~$\pi_\RR$) denote the projection on the complex (resp.~real) coordinate of $\CC \times \RR$. To focus on the complex part of such three dimensional vector differences, we flatten them and we study the statistical distribution of the complex differences $\pi_\CC(\Pow_\alpha(n)-\Pow_\alpha(m))$ for all $m,n \in \Lambda_S$ such that $|\pi_\RR(n-m)| < 2\pi$. This condition is introduced for the points $m$ and $n$ to be on the same "copy" of $\CC^*$ in its universal cover $\CC$. This is not a constraint since we multiply all differences $\Pow_\alpha(n)-\Pow_\alpha(m)$ by a scaling factor going to infinity and evaluate the related measures on a compactly supported function: after rescaling, pairs of points failing to satisfy this condition uniformly give rise to differences in $\CC \times \RR$ escaping all compact subsets. Let $\phi,\psi : \NN \to \, ]0,+\infty[\,$ be two functions converging to $+\infty$, which we respectively call the \emph{scaling factor} and the \emph{renormalization factor}. Throughout this paper, we fix $\lambda \in [0,+\infty]$ and we assume the following convergence and formula
\begin{align*}
	& \frac{\phi(N)}{N^{1-\alpha}} \to \lambda \in [0,+\infty] \mbox{ as } N\to\infty \numberthis\label{eq:scale_formula},
	\\ &\psi(N)=\Big( \frac{N^{2-\alpha}}{\phi(N)} \Big)^2 \mbox{ for all } N\in\NN. \numberthis\label{eq:renormalization_formula}
\end{align*}
Compared to the case of the introduction, taking into account all directions of noncompactness in $S \subset \CC \times \RR$, the need for two new integer parameters $N'$ and $N''$ emerges. We are interested in the multi-index sequence of \emph{empirical pair correlation measures} whose formula is given for all $N,N',N'' \in \NN-\{0\}$ by
\begin{align*}
\R_{N,N',N''}^{\alpha,\Lambda} & = \frac{1}{(N'+N'')\psi(N)} \sum_{\substack{m,n \in \Lambda_S, \; n \neq m \\ |\pi_\RR(n-m)| < 2\pi \\ 0 < |\pi_\CC(m)|,|\pi_\CC(n)| \leq N \\ -2\pi N' \leq \pi_\RR(m), \pi_\RR(n) < 2 \pi N''}} \Delta_{\phi(N)(\pi_\CC(\Pow_\alpha(n)-\Pow_\alpha(m)))}
\\ & = \frac{1}{(N'+N'')\psi(N)} \sum_{\substack{m,n \in \Lambda, \, n \neq m \\ 0 < |m|,|n| \leq N}}  \, \sum_{\substack{ r \in \exp^{-1}(m), \; s \in \exp^{-1}(n) \\ |\Im(r)-\Im(s)| < 2 \pi \\ -2\pi N' \leq \Im(r), \Im(s) < 2\pi N''}} \Delta_{\phi(N)(\exp(\alpha s) - \exp(\alpha r))}.\numberthis\label{eq:def_empirical_pair_cor_riemsurf}
\end{align*}
Let $\covol_{\vec{\Lambda}}$ be the covolume of $\vec{\Lambda}$, i.e.~the area of any fundamental parallelogram of $\vec{\Lambda}$. Set $\rho_{\alpha,\vec{\Lambda},\lambda}$ the nonnegative measurable function of formula
$$
\rho_{\alpha,\vec{\Lambda},\lambda} : z \mapsto \left\{
\begin{array}{cll}
	\displaystyle 0 & \mbox{ if } & \lambda = +\infty,\vspace{2mm}\\
	
	\displaystyle \frac{\pi}{\alpha^2(2-\alpha)\covol_{\vec{\Lambda}}^2} & \mbox{ if } & \lambda=0, \vspace{2mm}\\
	
	\displaystyle \frac{\alpha ^\frac{2}{1-\alpha}}{(1-\alpha) \covol_{\vec{\Lambda}}} \Big(\frac{|z|}{\lambda}\Big)^{-\frac{4-2\alpha}{1-\alpha}} \sum_{\substack{p\in\vec{\Lambda} \\ |p| \leq \frac{|z|}{\alpha\lambda}}} |p|^{\frac{2}{1-\alpha}} & \mbox{ if } & \lambda \in \, ]0,+\infty[\, .
\end{array} \right.
$$
The next two results will be proven at the end of Section \ref{sec:remove_branch_cut}.
\begin{theorem}\label{th:main}
We have the vague convergence, as $\min \{N, \, N'+N''\} \to\infty$,
$$
\R_{N,N',N''}^{\alpha,\Lambda} \weakstar \rho_{\alpha,\vec{\Lambda},\lambda} \, \Leb_\CC.
$$
\end{theorem}
\begin{figure}[ht]
	\centering
	\includegraphics[height=9cm]{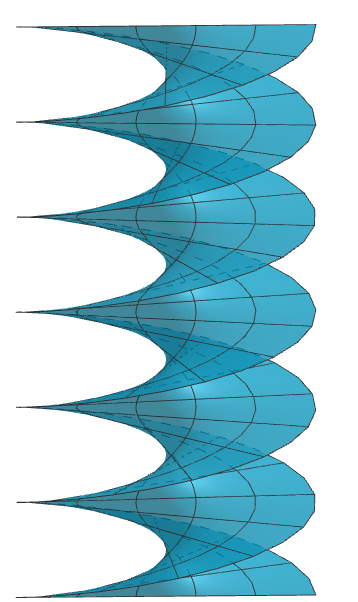}
	\includegraphics[height=9cm]{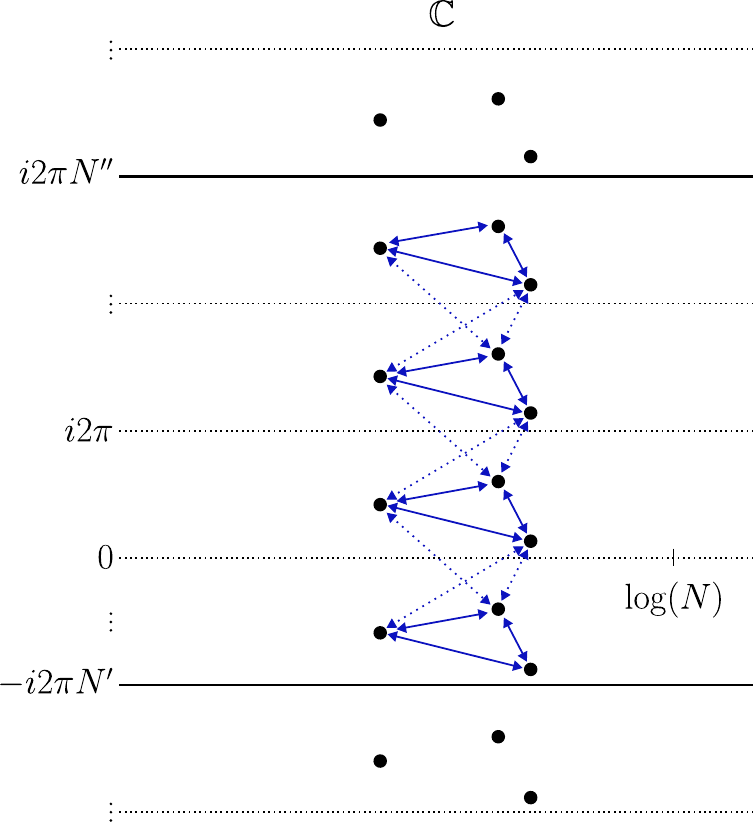}
	\caption{On the left, the helicoidal Riemann surface $S$. On the right, an illustration of the points (dots) $r \in \exp^{-1}(\Lambda) \subset \CC$. The (dotted and plain) two-headed arrows correspond to pairs of grid points appearing in the definition of the empirical pair correlation measure $\R_{N,N',N''}^{\alpha,\Lambda}$. The distinction between dotted and plain two-headed arrows will be explained before Theorem \ref{th:cv__complex_correlations}.}
	\label{fig:grid_in_riem_surf}
\end{figure}

For an error term in this convergence, see Remark \ref{rk:error_term_main}. In the case $\alpha \in \, ]0,1[ \, \cap \QQ$, we write its irreducible form $\alpha = \frac{a}{b}$ and obtain the following result.
\begin{theorem}\label{th:main_rat}
We have the vague convergence, as $N \to\infty$,
$$
\R_{N,0,b}^{\frac{a}{b},\Lambda} \weakstar \rho_{\frac{a}{b},\Lambda} \, \Leb_\CC.
$$
\end{theorem}

\begin{remark} Theorem \ref{th:main_rat} is not an immediate consequence of Theorem \ref{th:main} since $N'+N''=0+b$ does not go to infinity.
\end{remark}

\subsection{Separation into levels}\label{ssec:levels}
We use the notation $\RR_+=[0,+\infty[ \,$. For every real number $\beta$, every integer $k$ and every nonzero complex number $z$, we begin by defining the \emph{level-$k$ $\beta$ power of $z$} as
$$
z^{[\beta,k]} = |z|^\beta e^{i\beta \omega_k}, \; \mbox{ where $\omega_k$ is the representative in $[2\pi k, \, 2\pi(k+1)[\,$ of $\arg(z)$}.
$$
In other words, for every $z \in \CC-\RR_+$, we set $z^{[\beta,k]} = e^{\beta (\log(z)+i2\pi k)}$, where the map $\log : \CC-\RR_+ \mapsto \CC$ is the branch of the logarithm with branch cut $\RR_+$ and verifying $\log(-1)=i\pi$, and we extend this definition to $\CC^*$ in an "upper" continuous way, namely when $\Im(z) \geq 0$. For the particular case $k=0$, we use the notation $z^\beta=z^{[\beta,0]}$. This nonstandard choice of branch cut is handy for the following formula: for all $z,z' \in \CC^*$ and all $k\in\ZZ$,
$$ \frac{z^{[\beta,k]}}{z'^{[\beta,k]}} = \Big( \frac{z}{z'}\Big)^\beta \mbox{ or } \Big( \frac{z}{z'}\Big)^{[\beta,-1]},$$
depending on the sign of the difference $\omega-\omega'$ of the argument representatives $\omega$ of $z$ and $\omega'$ of $z'$, both taken in $[0,2\pi[\,$. In comparison, taking the principal branch of the logarithm to define these power functions would have required to separate between $3$ cases, whether the difference $\omega - \omega'$ belongs to $]-2\pi, -\pi]$, $]-\pi,\pi]$ or $]\pi, 2\pi]$. With the formula $z^\beta = e^{\beta \log(z)}$, we obtain the linear approximation, as $z \to 0$ with the restriction $\Im(z) \geq 0$,
\begin{equation}\label{eq:linear_approx_puiss_alph_half_space}
	(1+z)^\alpha = 1 + \alpha z + \bigO_\alpha(|z|^2).
\end{equation}
Note that the image of $\CC^*$ by the level-$k$ $\beta$ power function $z \mapsto z^{[\beta,k]}$ is the semi open angular sector $\{ z \in \CC^* \, : \, \arg(z) \in [2\pi k\beta, 2\pi (k+1)\beta[ \mod 2\pi \}$, in other words the sector of angle $2\beta\pi$ centred at the argument $2\pi(k+\frac{1}{2})\beta \mod 2\pi$.

We define the multi-index sequence of \emph{level separated empirical pair correlation measures} by
\begin{equation}\label{eq:def_correl_pair_lvl}
	\R_{N,N',N''}^{\alpha,\Lambda,\lvl} = \frac{1}{(N'+N'')\psi(N)} \sum_{k=-N'}^{N''-1} \sum_{\substack{n, m \in \Lambda, \, n\neq m \\ 0 < |n|,|m|\leq N}} \Delta_{\phi(N)(n^{[\alpha,k]} - m^{[\alpha,k]})}.
\end{equation} 
In comparison to the definition $\R_{N,N',N''}^{\alpha,\Lambda}$ from the beginning of Section \ref{sec:lemmas}, in the measure $\R_{N,N',N''}^{\alpha,\Lambda,\lvl}$ we do not take into account pairs of points illustrated with dotted arrows in Figure \ref{fig:grid_in_riem_surf}. Recall that the scaling and renormalization factors $\phi$ and $\psi$ verify the convergence \eqref{eq:scale_formula} and the formula \eqref{eq:renormalization_formula}.

\begin{theorem}\label{th:cv__complex_correlations}
	We have the following vague convergence of positive measures, as $\min \{N, \, N'+N''\} \to\infty$,
	$$ \R_{N,N',N''}^{\alpha,\Lambda,\lvl} \weakstar \rho_{\alpha,\vec{\Lambda},\lambda} \, \Leb_\CC.$$
\end{theorem}

\begin{figure}[ht]
	\centering
	\scalebox{0.75}{
		\begin{adjustbox}{clip, trim=6.8cm 1.2cm 7.1cm 3.cm, max width=\textwidth}
			\includegraphics{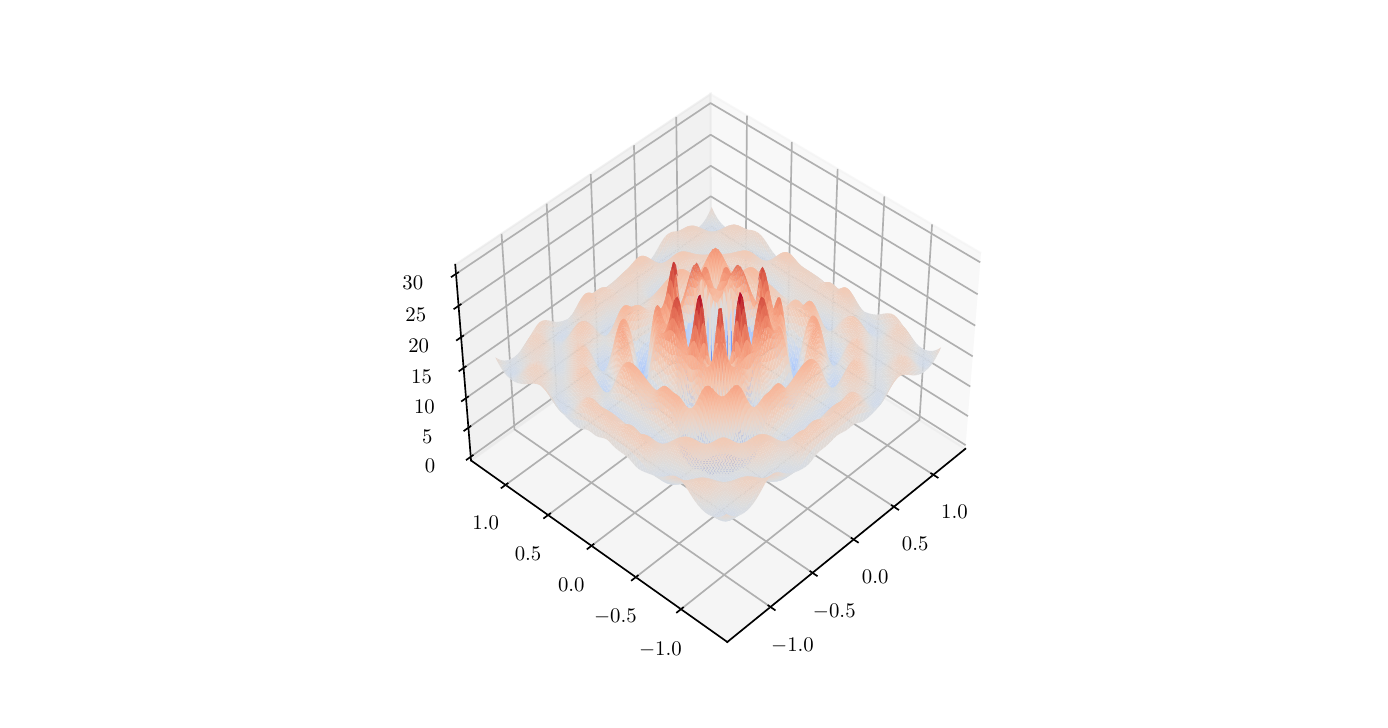}
		\end{adjustbox}
	}
	\scalebox{0.75}{
		\begin{adjustbox}{clip, trim=6.9cm 2cm 7cm 2.8cm, max width=\textwidth}
			\includegraphics{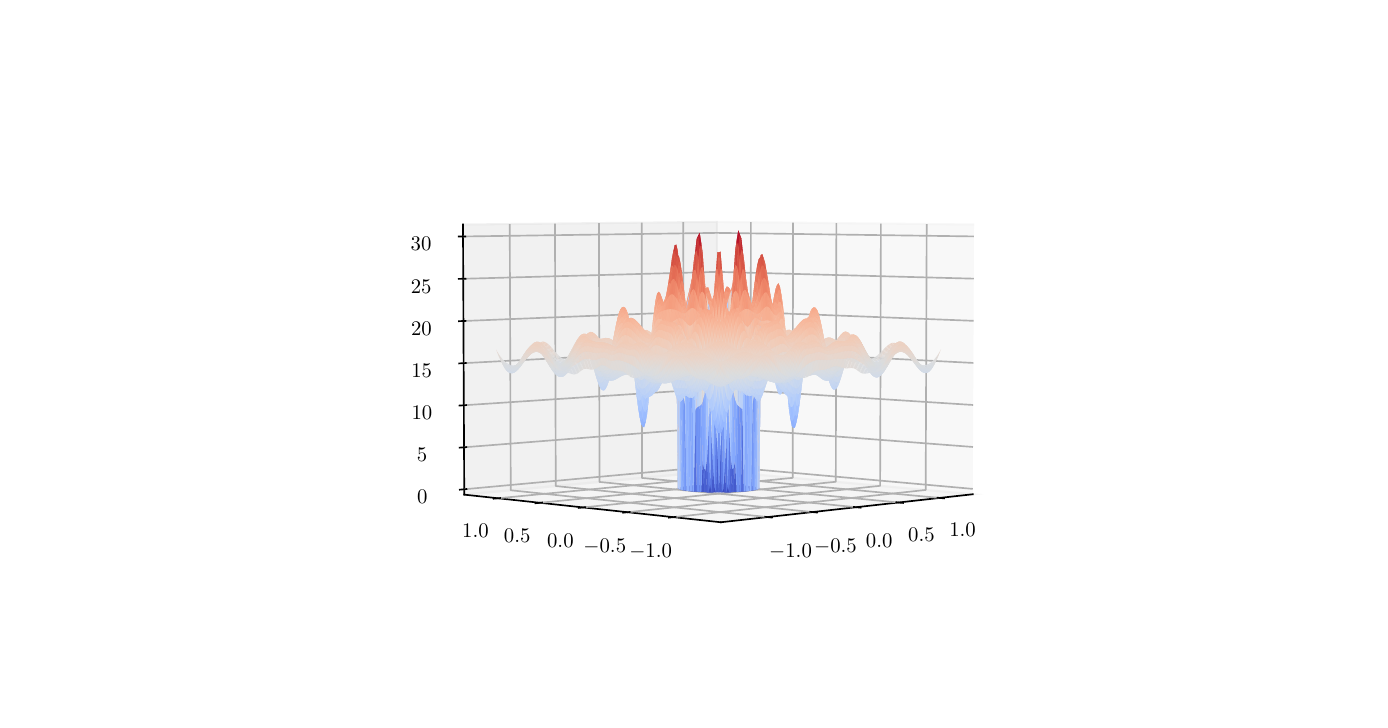}
		\end{adjustbox}
	}
	\caption{The empirical distribution obtained for the measure $\R_{N,N',N''}^{\frac{1}{3},\ZZ[i],\lvl}$ with $N=70$ and $N'+N''=3$ in the case $\lambda=1$, using a smoothing process of the library SciPy of Python.}
	\label{fig:emp_density_alpha13_exotic_N70}
\end{figure}

A qualitative illustration of this convergence is shown by comparing Figure \ref{fig:emp_density_alpha13_exotic_N70} to Figure \ref{fig:th_density_alpha13}, in the exotic case $\lambda=1$. Since the modulus function $|\cdot|$ from $\CC$ to $\RR_+$ is continuous and proper and since the function $\rho_{\alpha,\vec{\Lambda},\lambda}$ is invariant under rotation, the hypotheses of Theorem \ref{th:cv__complex_correlations} also imply the vague convergence, as the minimum $\min\{N,N'+N''\} \to \infty$,
$$
\frac{1}{(N'+N'')\psi(N)} \sum_{k=-N'}^{N''-1} \sum_{\substack{n, m \in \Lambda, \, n\neq m \\ 0 < |n|,|m|\leq N}} \Delta_{\phi(N)|n^{[\alpha,k]} - m^{[\alpha,k]}|} \weakstar 2\pi r \rho_{\alpha,\vec{\Lambda},\lambda}(r) dr.
$$
As an illustration of the latter convergence, a radial profile is drawn on Figure \ref{fig:radial_profile_alpha13_differentN}, in the exotic case $\lambda=1$.
\begin{figure}[ht]
	\centering
	\scalebox{0.98}{
	\begin{adjustbox}{clip, trim=2.2cm 0.7cm 2.2cm 1.2cm, max width=\textwidth}
		\includegraphics{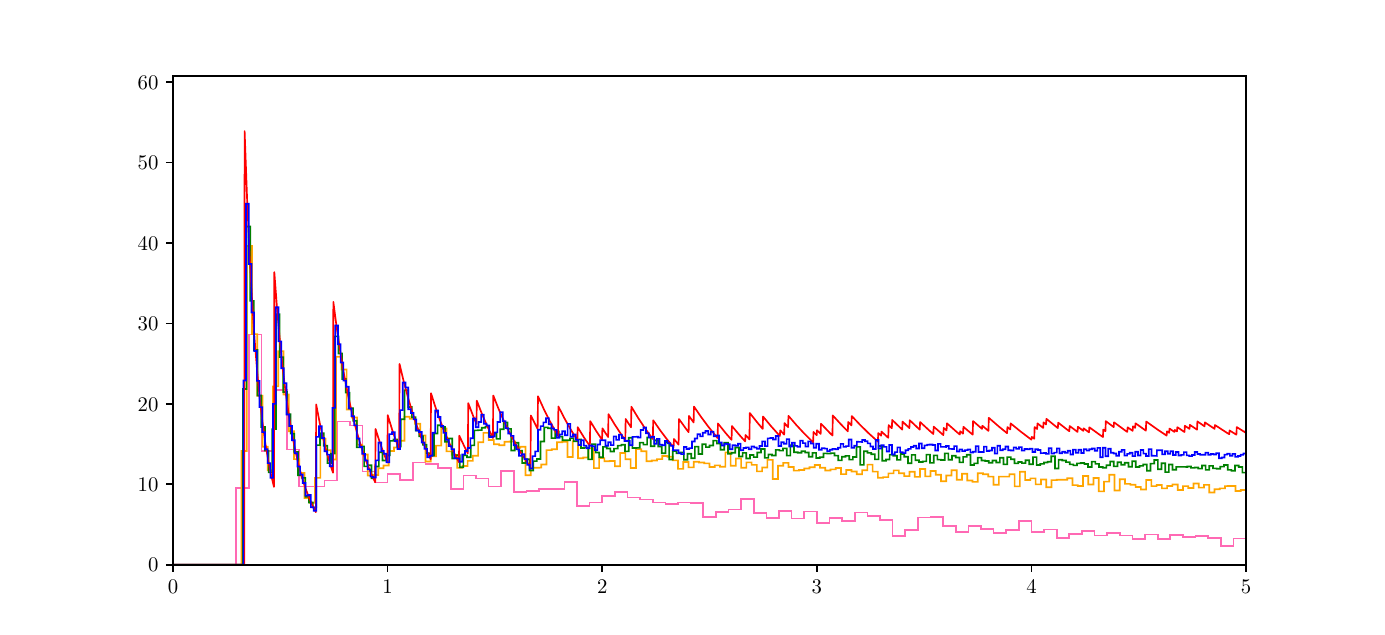}
	\end{adjustbox}
	}
	\caption{The empirical radial distribution of $\R_{N,N',N''}^{\frac{1}{3},\ZZ[i],\lvl}$ for $N'+N''=3$ and different values of $N$ ($N=10$ in \textcolor{magenta}{pink}, $N=30$ in \textcolor{orange}{orange}, $N=50$ in \textcolor{green}{green}, and $N=80$ in \textcolor{blue}{blue}) using the scaling factor $N\mapsto N^{\frac{2}{3}}$ (and renormalization factor $N\mapsto N^2$), and the limit density $r \mapsto \rho_{\frac{1}{3},\ZZ[i],1}(r)$ (in \textcolor{red}{red}).}
	\label{fig:radial_profile_alpha13_differentN}
\end{figure}

We denote by $\diam_{\vec{\Lambda}}$ the minimal diameter over all fundamental parallelograms of $\vec{\Lambda}$ and by $\sys_{\vec{\Lambda}}$ the \emph{systole} (or \emph{Minkowski's first minimum}) of the $\ZZ$-lattice $\vec{\Lambda}$, that is to say
$$\sys_{\vec{\Lambda}} = \min\{|p| \, : \, p \in \vec{\Lambda}, p \neq 0\} > 0.$$
We mention that the diameter $\diam_{\vec{\Lambda}}$ is comparable to the quantity $\frac{\covol_{\vec{\Lambda}}}{\sys_{\vec{\Lambda}}}$ thanks to the second theorem of Minkowski.
\begin{remark}{\rm
In the exotic case $\lambda \in \, ]0,+\infty[\,$, one can notice that we have $\rho_{\alpha,\vec{\Lambda},\lambda}=0$ on the open disk $C(0,\alpha\lambda \sys_{\vec{\Lambda}})$. This property is called a \emph{level repulsion phenomenon}. The fact that the radius $\alpha\lambda \sys_{\vec{\Lambda}}$ of this level repulsion disk converges to $+\infty$ as $\lambda \to +\infty$ can be interpreted as a continuity result between the cases $\lambda \in \, ]0,+\infty[\,$ and $\lambda=+\infty$. Such a continuity observation may also be made between the cases $\lambda \in \, ]0,+\infty[\,$ and $\lambda=0$, since Gauss counting argument (more precisely, its version for $\beta=\frac{2}{1-\alpha}$ stated in Lemma \ref{lem:sum_powers_grid}) indicates that, for all $\lambda \in \, ]0,+\infty[\,$,
$$ \rho_{\alpha,\vec{\Lambda},\lambda}(z) \underset{|z|\to \infty}{\longrightarrow} \frac{\pi}{\alpha^2(2-\alpha)\covol_{\vec{\Lambda}}^2}.$$
}
\end{remark}

\begin{remark}{\rm
Notice that $\rho_{\alpha,\vec{\Lambda},\lambda}$ is rotation invariant and, if $\lambda \in \, ]0,+\infty[\,$, the points of discontinuity of $\rho_{\alpha,\vec{\Lambda},\lambda}$ constitute the union of circles $\bigcup_{p\in\vec{\Lambda}-\{0\}} C(0, \alpha\lambda|p|)$. By comparison, extending the definition of $\R_{N,N',N''}^{\alpha,\Lambda,\lvl}$ to the simplistic case $\alpha=1$, choosing the scaling factor $N\mapsto 1$ (hence $\lambda=1$) and the renormalization factor $N \mapsto N^2$, a standard Gauss argument and a Riemann sum approximation grants the vague convergence, as $N \to \infty$,
$$ \R_{N,N',N''}^{1,\Lambda,\lvl} = \frac{1}{N^2} \sum_{\substack{n, m \in \Lambda, \, n\neq m \\ 0 < |n|,|m|\leq N}} \Delta_{n-m} \weakstar \frac{\pi}{\covol_{\vec{\Lambda}}} \sum_{p\in\vec{\Lambda}-\{0\}} \Delta_p.$$
In particular, the limit measure is not rotation invariant: we lose some symmetry in this extreme case $\alpha=1$.
}
\end{remark}

\begin{remark}{\rm
Upon an appropriate rescaling in terms of $\alpha$, a continuity statement can be made between the cases $\alpha \in \, ]0,1[\,$ and $\alpha=0$. We impose the scaling factor $\phi(N)=N^{1-\alpha}$ (hence $\lambda=1$) for this remark. Up to rotation, we can assume that the grid $\Lambda$ contains no nonzero point on the branch cut $\RR_+$ of the $\log$ function involved in the definition of $\alpha$-powers with levels. For all $k\in\ZZ$, all $n,m$ nonzero grid points in $\Lambda$ and all integer $N \in\NN$, notice that we have the convergence, as $\alpha \to 0^+$,
\begin{equation}\label{eq:link_alpha_power_log}
	\frac{1}{\alpha} N^{1-\alpha} (n^{[\alpha,k]}-m^{[\alpha,k]}) \longrightarrow N(\log(n)-\log(m)).
\end{equation}
We set
$$
\R_{N}^{\Lambda,\log} = \frac{1}{N^2} \sum_{\substack{n, m \in \Lambda, \, n \neq m \\ 0 < |n|,|m|\leq N}} \Delta_{N(\log(n)-\log(m))}.
$$
which is (up to the choice of a branch cut for the logarithm function) the empirical pair correlation measure studied in \cite[§~3]{paulinparkko2024_pairs_complexlog} for logarithm of grid points. Using Theorem \ref{th:cv__complex_correlations} and the fact that $z \mapsto \frac{z}{\alpha}$ is continuous and proper for the top convergence arrow, the convergence \eqref{eq:link_alpha_power_log} for the left-hand convergence arrow, and the dominated convergence theorem for the right-hand convergence arrow, we obtain the following diagram of vague convergence:
$$
\begin{array}{ccc}
	(z \mapsto \frac{z}{\alpha})_* \R_{N,N',N''}^{\alpha,\Lambda,\lvl} & \underset{\min(N,N'+N'')\to \infty}{\overset{\hbox{\hspace{0.15cm}$*$}}{\scalebox{2}[1]{$\rightharpoonup$}}} & (z \mapsto \frac{z}{\alpha})_* \rho_{\alpha,\vec{\Lambda},\lambda}\Leb_\CC = \alpha^2 \rho_{\alpha,\vec{\Lambda},\lambda}(\alpha z) \, dz \vspace{2mm}
	\\ \vcenter{\hbox{\scalebox{0.9}{$\overset{\alpha}{\underset{0^+}{\downarrow}}$}\vspace{0.17cm}}} \scalebox{1}[2]{$\downharpoonright$}* & & \vcenter{\hbox{\scalebox{0.9}{$\overset{\alpha}{\underset{0^+}{\downarrow}}$}\vspace{0.17cm}}} \scalebox{1}[2]{$\downharpoonright$}* \vspace{2mm}
	\\ \R_{N}^{\Lambda,\log} & &  \frac{|z|^4}{\covol_{\vec{\Lambda}}} \sum_{\substack{p\in\vec{\Lambda} \\ |p| \leq |z|}} |p|^2 \, dz.
\end{array}
$$
The bottom convergence arrow missing to this diagram has been proven in \cite[Theo.~3.1]{paulinparkko2024_pairs_complexlog}.
}
\end{remark}

In order to state an effective version of Theorem \ref{th:cv__complex_correlations}, we will use the space $C_c^1(\CC)$ of continuously differentiable functions of two real variables $f:\CC \to \CC$, with the standard notations $\|f\|_\infty=\sup_{z\in\CC}|f(z)|$ and $\|df\|_\infty = \sup_{z\in\CC}\|df(z)\|$, where $\|\cdot\|$ is the operator norm on the space of $\RR$-linear applications from $\CC$ to $\CC$. We use Landau's notation: for two sets of parameters $\P$ and $\P'$ with $\P' \subset \P$, for functions $F,G : \NN \mapsto \CC$ depending on (at least) the parameters in $\P'$, we write $F(N)=\bigO_{\P'}(G(N))$ if there exists some constant $c_{\P'} >0$, depending only on $\P'$, and some integer $N_0$, depending on all the parameters in $\P$, such that, for all $N \geq N_0$, we have the inequality $|F(N)| \leq c_{\P'} |G(N)|$. In our study, each time we will use Landau's notation, a test function $f \in C_c^1(\CC)$ will have been fixed, a bound $A$ on the size of its support will have been taken (i.e.~$\supp f \subset D(0,A)$) and our sets of parameters will always be $\P=\{\alpha, \Lambda, \phi, \psi, A\}$ and $\P'=\{\alpha\}$, $\{\Lambda\}$ or $\{\alpha, \Lambda\}$ (hence using the notation $\bigO_\alpha$, $\bigO_{\Lambda}$ or $\bigO_{\alpha,\Lambda}$). It is important to keep in mind that the rank $N_0$ may only depend on the parameters in $\P$. In particular, it does not depend on the parameters $N'$, $N''$, $\|f\|_\infty$, $\|df\|_\infty$, nor on any other index temporarily fixed in the proof of a lemma or a theorem.

For all $f \in C_c^1(\CC)$ and $A>1$, if $\lambda=+\infty$ we set $\Err_{{Th.\ref{th:effective_cv_complex_correlations}}}(\alpha,\Lambda,f,A)=0$, and otherwise we define
{\footnotesize
\begin{align*}
	& \Err_{{Th.\ref{th:effective_cv_complex_correlations}}}(\alpha,\Lambda,f,A)
	\\[2mm] = \; & \left\{
	\begin{array}{ll}
		\bigO_{\alpha,\Lambda} \big( A^4(\|f\|_\infty + \|df\|_\infty) \big( \frac{\phi(N)}{N^{1-\alpha}} + \frac{1}{N^\alpha \phi(N)} + \frac{1}{N'+N''} \big) \big) & \mbox{if } \lambda=0,  \vspace{3mm}
		\\ \bigO_{\alpha,\Lambda} \big((\|f\|_\infty + \|df\|_\infty)(\lambda+\frac{1}{\lambda})^{\frac{10-8\alpha}{1-\alpha}} \big( A^\frac{8-6\alpha}{1-\alpha} |\frac{\phi(N)}{\lambda N^{1-\alpha}} - 1| + \frac{A^4}{N} + \frac{A^2}{N'+N''} \big) \big) & \mbox{if } \lambda\in \, ]0,+\infty[ \, .
	\end{array}
	\right.
\end{align*}

\begin{theorem}\label{th:effective_cv_complex_correlations}
	Let $f \in C_c^1(\CC)$ and choose $A>1$ such that $\supp f \subset D(0,A)$.
	\begin{itemize}
		\item If $\lambda=+\infty$, then there exists an integer $N_0$ which depends on $\alpha$, $\Lambda$ and $A$, such that for all $N \geq N_0$ and all $N',N'' \in \NN$, we have $\R_{N,N',N''}^{\alpha,\Lambda,\lvl}(f)=0$.
		\item If $\lambda \in [0,+\infty[ \,$, as $N\to\infty$, we have
		$$
		\R_{N,N',N''}^{\alpha,\Lambda,\lvl}(f) = \int_\CC f(z) \rho_{\alpha,\vec{\Lambda},\lambda}(z) \, dz + \Err_{{Th.\ref{th:effective_cv_complex_correlations}}}(\alpha,\Lambda,f,A).
		$$
	\end{itemize}
\end{theorem}

\begin{remarks}\label{rk:effective_main_rat}
	{\rm
	\begin{itemize}
	\item By a standard approximation argument of a function in $C^0_c(\CC)$ by functions in $C^1_c(\CC)$, Theorem \ref{th:cv__complex_correlations} is an immediate consequence of Theorem \ref{th:effective_cv_complex_correlations}.
	\item For a version of the error term in Theorem \ref{th:effective_cv_complex_correlations} with explicit dependence on parameters of the grid $\Lambda$ (but not on the power parameter $\alpha$), see \cite{sayous2025PhD}.
	\item In the case $\alpha = \frac{a}{b} \in \QQ$, for all integers $N\in\NN$ and $k \in \ZZ-\{0\}$, we have the periodicity formula $\R_{N,0,k b}^{\frac{a}{b},\Lambda,\lvl}=\R_{N,0,b}^{\frac{a}{b},\Lambda,\lvl}$. This implies that we have, as $N\to\infty$,
	$$\R_{N,0,b}^{\frac{a}{b},\Lambda,\lvl} \weakstar \rho_{\alpha,\vec{\Lambda},\lambda} \Leb_\CC.$$
	\end{itemize}
	}
\end{remarks}

\subsection{Counting lemmas}\label{ssec:counting_lem}
Set $\sys_{\Lambda} = \min \{ |m| \, : \, m \in \Lambda, m \neq 0\}>0$. This quantity is not commonly used for studying grids, except when the grid is a lattice, in which case $\sys_{\Lambda}$ is the usual systole. It will be useful for many computations throughout this paper. The next lemma is a well known result which will be useful in order to explicitly compute the limit function $\rho_{\alpha,\vec{\Lambda},\lambda}$ as well as to bound error terms for Theorem \ref{th:effective_cv_complex_correlations}.
\begin{lemma}\label{lem:sum_powers_grid}
	For every real number $\beta \geq 0$, there exists a constant $C_{\beta, \Lambda} >0$ such that, for all $x\geq 1$,
	$$
	\Big| \sum_{\substack{m\in\Lambda \\ 0<|m|\leq x}} |m|^\beta - \frac{2\pi}{\covol_{\vec{\Lambda}}} \frac{x^{\beta+2}}{\beta+2} \Big| \leq C_{\beta, \Lambda} \, x^{\beta+1}.
	$$
	For every real number $\beta > -2$, we have the (less explicit) estimate, as $x\to\infty$,
	$$
	\sum_{\substack{m\in\Lambda \\ 0<|m|\leq x}} |m|^\beta = \frac{2\pi}{\covol_{\vec{\Lambda}}} \frac{x^{\beta+2}}{\beta+2} +\bigO_{\beta,\Lambda} \Big( x^{\beta+1} + 1 \Big).
	$$
	In the case $\beta=-2$, we have the following estimate, as $x\to+\infty$,
	$$
	\sum_{\substack{m\in\Lambda \\ 0<|m|\leq x}} \frac{1}{|m|^2} \sim \frac{2\pi}{\covol_{\vec{\Lambda}}} \log(x).
	$$
	For every real number $\beta < -2$, we have the convergence
	$$
	\sum_{m\in\Lambda-\{0\}} |m|^\beta < \infty.
	$$
\end{lemma}

\begin{proof}
	We recall Abel's summation formula: for every real sequence $(a_k)_{k\geq 1}$, all real numbers $1\leq x_0 \leq x$ and all function $f : [x_0, +\infty[\, \to \RR$ of class $C^1$ on $\, ]x_0, +\infty[$, we have the equality
	\begin{equation}\label{eq:abel_formula}
		\sum_{x_0\leq k\leq x} a_k f(k) =
		\big(\sum_{1 \leq k\leq x} a_k \big) f(x)
		- \big(\sum_{1 \leq k < x_0} a_k \big) f(x_0)
		- \int_{x_0}^x\big(\sum_{1 \leq k\leq t} a_k \big) f'(t) \, dt
	\end{equation}
	Let $x\geq 1$ and $\F$ be a closed fundamental parallelogram of $\vec{\Lambda}$ containing $0$ with minimal diameter.
	
	For the case $\beta=0$, we follow the standard Gauss counting argument. Set $A_x=\{ m \in \Lambda \, : \, 0<|m| \leq x\}$ and $B_x=\bigcup_{m\in A_x} (m+\F)$, so that we have the equality $\Leb_\CC(B_x) = \card(A_x) \covol_{\vec{\Lambda}}$. The definition of $\diam_{\vec{\Lambda}}$ yields the inclusions
	\begin{equation}\label{eq:inclusions_gauss_counting}
		\bar{D}(0,x-\diam_{\vec{\Lambda}}) \subset B_x \subset \bar{D}(0,x+\diam_{\vec{\Lambda}}),
	\end{equation}
	where the closed disk $\bar{D}(0,x-\diam_{\vec{\Lambda}})$ is empty if $x < \diam_{\vec{\Lambda}}$. Computing the Lebesgue measure of these disks gives
	\begin{equation}\label{eq:gauss_count_beta_0_alt}
		\frac{\pi}{\covol_{\vec{\Lambda}}} \max \{ 0, x-\diam_{\vec{\Lambda}}\}^2 \leq \card(A_x) \leq \frac{\pi}{\covol_{\vec{\Lambda}}}(x+\diam_{\vec{\Lambda}})^2
	\end{equation}
	which is even valid in the case $0 \leq x < 1$ and implies the lemma in the case $\beta=0$.
	
	Assume $\beta > 0$. Consider the sequence $(a_k = \card \{ m\in \Lambda \, : \, k-1 < |m| \leq k \})_{k\geq 1}$. We have the following inequalities
	$$
	\sum_{1 \leq k \leq x } a_k (k-1)^\beta \leq \sum_{\substack{m\in\Lambda \\ 0 < |m| \leq x}} |m|^\beta \leq \sum_{1 \leq k \leq \lceil x \rceil} a_k k^\beta.
	$$
	Let $\floor \cdot $ denote the lower integral part on $\RR$. Applying Abel's formula \eqref{eq:abel_formula} to $f:t\mapsto t^\beta$ then $f:t \mapsto (t-1)^\beta$ with $x_0=1$, together with the case $\beta=0$ to estimate $\sum_{1 \leq k \leq x}a_k =\card(A_{\lfloor x \rfloor})$, this proves the lemma in the case $\beta >0$.
	
	Assume $\beta \in \, ]-2,0[\,$. Then we have the inequalities
	\begin{equation}\label{eq:pre_abel_beta_neg}
		\sum_{2 \leq k \leq x } a_k k^\beta \leq \sum_{\substack{m\in\Lambda \\ 1 < |m| \leq x}} |m|^\beta \leq \sum_{2 \leq k \leq \lceil x \rceil} a_k (k-1)^\beta.
	\end{equation}
	Applying Abel's formula \eqref{eq:abel_formula} to $f:t\mapsto t^\beta$ then $f:t \mapsto (t-1)^\beta$ with $x_0=2$, this proves the estimate, as $x\to \infty$,
	$$\sum_{\substack{m\in\Lambda \\ 1 < |m| \leq x}} |m|^\beta = \frac{2\pi}{\covol_{\vec{\Lambda}}} \frac{x^{\beta +2}}{\beta+2} + \bigO_\beta \big( \frac{1+\diam_{\vec{\Lambda}}^2}{\covol_{\vec{\Lambda}}} x^{\beta+1} \big).$$
	Combining this with the inequality $\sum_{\substack{m\in\Lambda \\ 0 < |m| \leq 1}} |m|^\beta \leq \sys_{\Lambda}^\beta \frac{\pi(1+\diam_{\vec{\Lambda}}^2)}{\covol_{\vec{\Lambda}}}$ coming from Equation \eqref{eq:gauss_count_beta_0_alt}, the lemma is proven in the case $\beta \in \, ]-2,0[\,$.
	
	The case $\beta=-2$ directly comes from the inequalities \eqref{eq:pre_abel_beta_neg} and the same application of Abel's formula, since then the only diverging term is equivalent to, as $x \to \infty$,
	$$
	- \int_2^x \frac{\pi}{\covol_{\vec{\Lambda}}} t^2 f'(t) \sim \frac{2 \pi}{\covol_{\vec{\Lambda}}} \log(x)
	$$
	in both cases where $f$ is given by $t\mapsto \frac{1}{t^2}$ or by $t \mapsto \frac{1}{(t-1)^2}$.
	
	For the case $\beta<-2$, using the case $\beta=0$ from Equation \eqref{eq:gauss_count_beta_0_alt}, we can directly compute
	{\small
	\begin{align*}
		\sum_{\substack{m\in\Lambda \\ 0 < |m| \leq x}} |m|^\beta & = \sum_{\substack{m\in\Lambda \\ 0 < |m| \leq 3\diam_{\vec{\Lambda}}}} |m|^\beta + \sum_{\substack{m\in\Lambda \\ 3\diam_{\vec{\Lambda}} < |m| \leq x}} |m|^\beta
		\\ & \leq \bigO_{\beta, \Lambda}(1) + \frac{1}{\covol_{\vec{\Lambda}}} \int_{\CC-D(0,2\diam_{\vec{\Lambda}})} (|z|-\diam_{\vec{\Lambda}})^\beta \, dz = \bigO_{\beta, \Lambda}(1).
	\end{align*}
	which gives an upper bound independent of $x$ for the sum.
	}
\end{proof}

Another helpful tool is given in the next lemma: it will allow us to count grid points that are near a given straight line.
\begin{lemma}\label{lem:grid_pts_near_line}
	Let $g:\RR_+\to\RR_+$ be a nonnegative piecewise continuous function and set $L_g=\{x+iy \, : \, x \geq 0, y\in\RR \mbox{ and } |y| \leq g(x)\}$. Then, for all $N\in\NN$, we have the inequality
	$$ \card \big(\Lambda \cap D(0,N) \cap L_g \big) \leq 4 \sum_{x=1}^N \frac{(1+\diam_{\vec{\Lambda}})(\max_{[x-1,x]}g+\diam_{\vec{\Lambda}})}{\covol_{\vec{\Lambda}}}.$$
\end{lemma}

\begin{proof}
	Fix $N\in\NN-\{0\}$. For every $x\in\{1,\ldots,N\}$, let $m_x$ denote the real number $\max_{[x-1,x]}g$ and consider the rectangle $R_x=[x-1,x]+i[-m_x,m_x]$. We have the inequality
	\vspace{-0.15cm}
	$$\card \big(\Lambda \cap D(0,N) \cap L_g \big) \leq \sum_{x=1}^N \card(\Lambda \cap R_x).$$
	\vspace{-0.7cm}
	\begin{figure}[h!]
		\centering
		\includegraphics[height=6.7cm]{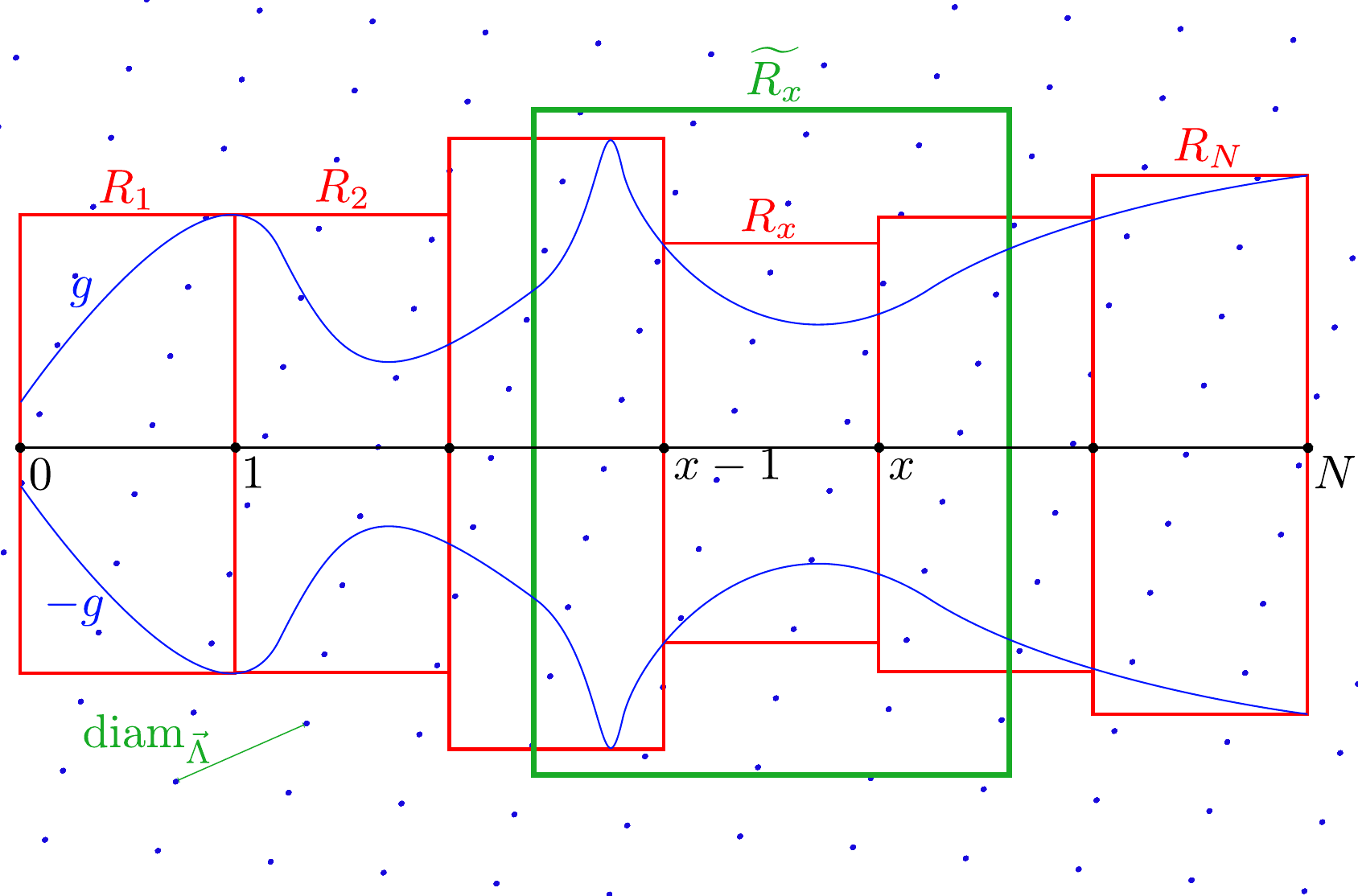}
		\label{fig:proof_points_near_line}
	\end{figure}

	\vspace{-0.15cm}
	\noindent For each $x$, let us denote by $\widetilde{R_x}$ the $\diam_{\vec{\Lambda}}$-neighbourhood of $R_x$ for the infinity norm $\|z\|_\infty = \max\{ |\Re(z)|, |\Im(z)|\}$ (so that $\widetilde{R_x}$ is a rectangle, see Figure \ref{fig:proof_points_near_line}). Using Gauss counting argument, the inequality between the Euclidean norm and the infinity norm then yields the inequality, for all $x\in\{1, \ldots, N\}$,
	$$\card(\Lambda \cap R_x) \covol_{\vec{\Lambda}} \leq \vol(\widetilde{R_x}) = (1+2\diam_{\vec{\Lambda}})(2m_x+2\diam_{\vec{\Lambda}}).$$
	Summing over $x\in\{1, \ldots, N\}$ proves the lemma.
\end{proof}

\subsection{Symmetry lemma}\label{ssec:symmetry_lem}
By the change of variable $p=n-m$, we can rewrite the definition \eqref{eq:def_correl_pair_lvl} as follows
\begin{equation}\label{eq:def_correl_pair_m,p}
	\R_{N,N',N''}^{\alpha,\Lambda,\lvl} = \frac{1}{(N'+N'')\psi(N)} \sum_{k=-N'}^{N''-1} \sum_{p\in\vec{\Lambda}-\{0\}}
	\sum_{\substack{m \in \Lambda\\ 0 < |m+p|,|m|\leq N}} \Delta_{\phi(N)((m+p)^{[\alpha,k]} - m^{[\alpha,k]})}
\end{equation} 
For any number $z\in\CC^*$, recall the notation $z^\alpha = z^{[\alpha,0]}$ for its level-$0$ $\alpha$ power. Let $\theta : \CC^* \to \RR$ denote the projection of the argument function onto $[0, 2\pi[ \,$. For all nonzero complex numbers $z,z'$, the definition of their level-$k$ $\alpha$ power yields the formula $\frac{z^{[\alpha,k]}}{z'^{[\alpha,k]}}=(\frac{z}{z'})^{[\alpha,l]}$ where $l=\lfloor\frac{\theta(z)-\theta(z')}{2\pi}\rfloor = 0 \mbox{ or} -1$ depending on the sign of $\theta(z)-\theta(z')$, independently of $k$. Set
\begin{align*}
	I_N^+ & = \{ (m,p) \in \Lambda \times (\vec{\Lambda}-\{0\}) \, : \, 0 < |m|, |m+p| \leq N \mbox{ and } \theta(m+p) > \theta(m) \} \\
	\mbox{and } I_N^-& = \{ (m,p) \in \Lambda \times (\vec{\Lambda}-\{0\}) \, : \, 0 < |m|, |m+p| \leq N \mbox{ and } \theta(m+p) < \theta(m)\}.
\end{align*}
In other words, putting aside the case $\theta(m+p)=\theta(m)$ for now, the set $I_N^+$ (resp.~$I_N^-$) contains the indices $(m,p)$ in Equation \eqref{eq:def_correl_pair_m,p} verifying the formula, for all $k\in\ZZ$,
$$
\frac{(m+p)^{[\alpha,k]}}{m^{[\alpha,k]}}=(1+\frac{p}{m})^\alpha \, \mbox{ \big(resp.~}\frac{(m+p)^{[\alpha,k]}}{m^{[\alpha,k]}}=(1+\frac{p}{m})^{[\alpha,-1]} \mbox{\big)}.
$$
Let $\R_{N,N',N''}^{\alpha,\Lambda,+}$ (resp.~$\R_{N,N',N''}^{\alpha,\Lambda,-}$) denote the part of $\R_{N,N',N''}^{\alpha,\Lambda,\lvl}$ with indices in $I_N^+$ (resp.~in $I_N^-$) in Equation \eqref{eq:def_correl_pair_m,p}. One can notice a one-to-one correspondence between $I_N^+$ and $I_N^-$ given by the map $(m,p) \mapsto (m+p,-p)$. This yields the formula 
\begin{equation}\label{eq:symmetry}
	\R_{N,N',N''}^{\alpha,\Lambda,-}= (z\mapsto-z)_*\R_{N,N',N''}^{\alpha,\Lambda,+}.
\end{equation}
The next lemma indicates that the contribution of the indices $(m,p)$ which do not belong to $I_N^+$ nor $I_N^-$ is negligible. Combined with the formula \eqref{eq:symmetry}, we will be able to derive the vague convergence of $\R_{N,N',N''}^{\alpha,\Lambda,\lvl}$ from the one of $\R_{N,N',N''}^{\alpha,\Lambda,+}$.
\begin{lemma}\label{lem:negligible_diag_points}
	Let $f \in C_c^1(\CC)$ and choose $A>1$ such that $\supp f \subset D(0,A)$. We have the estimate, as $N \to \infty$,
	$$
	\R_{N,N',N''}^{\alpha,\Lambda,\lvl}(f) = \R_{N,N',N''}^{\alpha,\Lambda,+}(f)+\R_{N,N',N''}^{\alpha,\Lambda,-}(f) + \bigO_{\alpha,\Lambda}\Big( \frac{N\|f\|_\infty}{\psi(N)} \Big( \Big(\frac{A N^{1-\alpha}}{\phi(N)}\Big) + 1 \Big)^2 \Big)
	$$
\end{lemma}

\begin{proof}
	The difference $\R_{N,N',N''}^{\alpha,\Lambda,\lvl}(f) - \big(\R_{N,N',N''}^{\alpha,\Lambda,+}(f)+\R_{N,N',N''}^{\alpha,\Lambda,-}(f)\big)$ is
	\begin{equation}\label{eq:correl_pair_diag_points}
		\frac{1}{(N'+N'')\psi(N)} \sum_{k=-N'}^{N''-1} \sum_{p\in\vec{\Lambda}-\{0\}}
		\sum_{\substack{m \in \Lambda\\ 0 < |m+p|,|m|\leq N \\ \theta(m+p)=\theta(m)}} f(\phi(N)((m+p)^{[\alpha,k]} - m^{[\alpha,k]})).
	\end{equation}
	Fix $k\in\ZZ$. Our goal is then to count pairs of points $(m,p) \in \Lambda \times (\vec{\Lambda}-\{0\})$ verifying the inequalities $0<|m|,|m+p| \leq N$, the equality of arguments $\theta(m+p)=\theta(m)$ and the inequality $|\phi(N)((m+p)^{[\alpha,k]} - m^{[\alpha,k]})| \leq A$. We denote by $I_{N,A}^=$ the set of such indices $(m,p)$ (which indeed does not depend on $k$ thanks to the formula $z^{[\alpha,k]}=e^{i 2\pi k\alpha} z^\alpha$). Let $(m,p) \in I_{N,A}^=$. We denote by $\omega = \theta(m) = \theta(m+p)$ their common argument in $[0,2\pi[\,$. The function $z \mapsto z^{[\alpha,k]}$ is regular when we restrict it to the segment $[m,m+p]$: the complex valued function $g:t \mapsto (m+pt)^{[\alpha,k]}=(|m|+t|p|)^\alpha e^{i\alpha (\omega+2\pi k)}$ is differentiable and its derivative is given by $g':t \mapsto \frac{\alpha e^{i\alpha (\omega+2\pi k)} |p|}{(|m|+t|p|)^{1-\alpha}}$. It is minimal in modulus when $t=1$, for which we have $|g'(1)|= \frac{\alpha |p|}{|m+p|^{1-\alpha}}$. The mean value inequality then grants us
	$$ \frac{A}{\phi(N)} \geq |(m+p)^{[\alpha,k]} - m^{[\alpha,k]}| = |g(1)-g(0)| \geq \frac{\alpha |p|}{|m+p|^{1-\alpha}}.$$
	From this, we derive the main inequality that we will use to count such pairs of points, namely
	\begin{equation}\label{eq:bnd_diag_p}
		|p| \leq \frac{AN^{1-\alpha}}{\alpha \phi(N)}.
	\end{equation}
	As $N\to \infty$, Equation \eqref{eq:gauss_count_beta_0_alt} indicates that there are only $\bigO_\alpha \big( \frac{1}{\covol_{\vec{\Lambda}}} \big( \frac{A N^{1-\alpha}}{\alpha \phi(N)} + \diam_{\vec{\Lambda}} \big)^2 \big)$ points $p \in \vec{\Lambda}$ verifying the inequality \eqref{eq:bnd_diag_p}. Let us fix such a point $p$. Then, for the points $0$, $m$ and $m+p$ to be aligned, the nonzero grid point $m+p$ has to be chosen on the ray from $0$ to $p$. Since moreover it has to be in the closed disk $\bar{D}(0,N)$, there are at most $\frac{N}{\sys_{\vec{\Lambda}}}$ ways to choose the point $m+p$. This counting argument yields, as $N \to \infty$,
	{\footnotesize
	\begin{equation}\label{eq:estimate_pts_diag}
		\card(I_{N,A}^=) = \bigO_\alpha \Big( \frac{N}{\sys_{\vec{\Lambda}} \covol_{\vec{\Lambda}}} \Big( \Big(\frac{A N^{1-\alpha}}{\alpha \phi(N)}\Big) + \diam_{\vec{\Lambda}} \Big)^2 \Big) = \bigO_{\alpha,\Lambda} \Big( N \Big( \Big(\frac{A N^{1-\alpha}}{\alpha \phi(N)}\Big) + 1 \Big)^2  \Big).
	\end{equation}
	}The triangle inequality applied to Equation \eqref{eq:correl_pair_diag_points} gives the estimate stated in the lemma.
\end{proof}

\begin{remark}\label{rk:negligible_diag_points_case_lambda_0}
	{\rm
	Since the renormalization factor is given by $\psi(N) = (\frac{N^{2-\alpha}}{\phi(N)})^2$, in the case $\lambda=0$ of Theorem \ref{th:effective_cv_complex_correlations}, the estimate of Lemma \ref{lem:negligible_diag_points} becomes
	$$
	\R_{N,N',N''}^{\alpha,\Lambda,\lvl}(f) = \R_{N,N',N''}^{\alpha,\Lambda,+}(f)+\R_{N,N',N''}^{\alpha,\Lambda,-}(f) + \bigO_{\alpha,\Lambda}\Big(\frac{A^2 \|f\|_\infty}{N}\Big).
	$$
	}
\end{remark}

\subsection{Linear approximation}\label{ssec:linear_approx_lem}
Thanks to Lemma \ref{lem:negligible_diag_points} and the symmetry formula \eqref{eq:symmetry}, for every $f \in C_c^1(\CC)$, we can focus on the asymptotic behaviour of the sequence $(\R_{N,N',N''}^{\alpha,\Lambda,+}(f))_{N,N',N''\in\NN}$, whose formula can be rewritten
\begin{equation}\label{eq:def_correl_pos_m,p}
	\R_{N,N',N''}^{\alpha,\Lambda,+}(f) = \frac{1}{(N'+N'')\psi(N)} \sum_{k=-N'}^{N''-1} \sum_{(m,p)\in I_N^+} f\big(\phi(N)m^{[\alpha,k]}((1+\frac{p}{m})^\alpha - 1)\big).
\end{equation}
Define another sequence of positive measures by
$$\mu_{N,N',N''}^+ = \frac{1}{(N'+N'')\psi(N)} \sum_{k=-N'}^{N''-1} \sum_{(m,p)\in I_N^+} \Delta_{\phi(N) \frac{\alpha p}{m^{[1-\alpha,k]}}}.$$

The next result is a linear approximation lemma. 

\begin{lemma}\label{lem:complex_linear_approx}
	Let $f \in C_c^1(\CC)$ and choose $A>1$ such that $\supp f \subset D(0,A)$. We assume that $\frac{\phi(N)}{N^{1-\alpha}} \underset{N\to\infty}{\longrightarrow} \lambda \in [0,+\infty[\,$. Then we have, as $N \to \infty$, 
	{\footnotesize
	\begin{align*}
	\R_{N,N',N''}^{\alpha,\Lambda,+}(f) - \mu_{N,N',N''}^+(f) = \bigO_{\alpha,\Lambda} \Big(\|df\|_\infty \big(\frac{A^4}{N^\alpha\phi(N)} + \frac{A^2\phi(N)}{N^{2-\alpha}}\big) + \|f\|_\infty \big(\frac{A^3}{N^\alpha\phi(N)}+\frac{\phi(N)^2}{N^{3-2\alpha}}\big) \Big).
	\end{align*}
	}
\end{lemma}

\begin{proof}
	Fix $k\in\ZZ$. For all $(m,p) \in I_N^+$, we want to bound from above the quantity
	\begin{equation}\label{eq:complex_linear_approx_m,p}
		\big| f\big(\phi(N)m^{[\alpha,k]} \big( \big(1+\frac{p}{m}\big)^\alpha - 1 \big)\big)\big) - f\big(\phi(N) \frac{\alpha p}{m^{[1-\alpha,k]}}\big) \big|.
	\end{equation}
	By the hypothesis $\supp f \subset D(0,A)$, in order for the latter quantity not to be equal to $0$, the index $(m,p)$ has to verify (at least) one of the two inequalities
	\vspace{-0.1cm}
	\begin{equation}\label{eq:complex_bounds_INA}
		|p| \leq \frac{A |m|^{1-\alpha}}{\alpha \phi(N)} \; \mbox{ or } \; \big|\big(1+\frac{p}{m}\big)^\alpha - 1\big| \leq \frac{A}{|m|^\alpha\phi(N)}.
	\end{equation}
	Let $I_{N,A}^+$ be the subset of $I_N^+$ consisting of such indices. Let $(m,p) \in I_{N,A}^+$. Note that the inverse of the map $z\mapsto z^\alpha$ is Lipschitz continuous on a small neighbourhood of $1=1^\alpha$ in the image of $z\mapsto z^\alpha$. This neighbourhood may be taken to be $D(1,1) \cap \{z \in \CC \, : \, \Im(z) \geq 0\}\cap (z\mapsto z^\alpha)(\CC^*)$, which is convex and where $(z\mapsto z^\alpha)^{-1}$ has its derivative's norm lesser than $\frac{2^{\frac{1}{\alpha}-1}}{\alpha}$. Then, as a consequence of Equation \eqref{eq:complex_bounds_INA}, since $\frac{A}{|m|^\alpha\phi(N)} \leq \frac{A}{\sys_\Lambda^\alpha \phi(N)} \to 0$ as $N\to \infty$, we have the estimate, as $N\to\infty$,
	\vspace{-0.15cm}
	\begin{equation}\label{eq:linear_approx_unif_conv_p/m}
			\big| \frac{p}{m} \big| = \bigO_\alpha\Big(\frac{A}{|m|^\alpha \phi(N)}\Big).
	\end{equation}
	(Recall that, thank to the definition of $\bigO_\alpha$, the latter estimate is uniform over any temporarily fixed variable, in particular over $(m,p) \in I_{N,A}^+$). One may notice that Equation \eqref{eq:linear_approx_unif_conv_p/m} implies that $m+p$ and $m$ are not independent grid points: they have to be close together since $p$ has to verify $|p| = \bigO_\alpha \big( \frac{A|m|^{1-\alpha}}{\phi(N)} \big)$. With the consequential estimate $|p| = \bigO_\alpha \big( \frac{AN^{1-\alpha}}{\phi(N)} \big)$, we use the Gauss counting argument from Equation \eqref{eq:gauss_count_beta_0_alt} (summing over $\vec{\Lambda}$ with $x=\bigO_\alpha \big( \frac{AN^{1-\alpha}}{\phi(N)} \big)$) to deduce a result that we will use twice in the remaining part of the proof: as $N\to\infty$, we have the estimate (uniformly for every grid point $m\in\Lambda$),
	\vspace{-0.15cm}
	\begin{equation}\label{eq:linear_approx_card_p_lambda=0}
		\card \{ p \in \vec{\Lambda}-\{0\} \, : \, (m,p) \in I_{N,A}^+\} = \bigO_{\alpha,\Lambda} \Big( \Big(\frac{AN^{1-\alpha}}{\phi(N)}+1\Big)^2\Big).
	\end{equation}
	Recall the linear approximation \eqref{eq:linear_approx_puiss_alph_half_space} as $z \to 0$ with the restriction $\Im(z) \geq 0$. In order to apply it to most fractions $z=\frac{p}{m}$, we have to take out the indices $(m,p)$ for which $\Im(\frac{p}{m}) < 0$ holds. For that matter, we first notice that for all $(m,p) \in I_{N,A}^+$, the inequality $\Im(\frac{p}{m}) < 0$ holds if, and only if, we have $\theta(m+p)-\theta(m) \in \, ]\pi,2\pi[ \,$ (since $\Im(\frac{p}{m}) = \Im(\frac{m+p}{m})$). We denote by $I_{N,A}^{\rm bad}$ the set of these indices. Then, by use of Equation \eqref{eq:linear_approx_unif_conv_p/m}, for all indices $(m,p) \in I_{N,A}^{\rm bad}$, we have the estimate, as $N\to\infty$,
	\vspace{-0.15cm}
	{\footnotesize
	\begin{align*}
		\big|\frac{p}{m}\big| = \big|\frac{m+p}{m}-1\big| & = \big||1+\frac{p}{m}|e^{i(\theta(m+p)-\theta(m))} - 1\big| = |e^{i(\theta(m+p)-\theta(m))} - 1| + \bigO_\alpha\Big(\frac{A}{|m|^\alpha \phi(N)}\Big)
		\\& = |e^{i\frac{\theta(m+p)-\theta(m)}{2}} - e^{-i\frac{\theta(m+p)-\theta(m)}{2}}| + \bigO_\alpha\Big(\frac{A}{|m|^\alpha \phi(N)}\Big)
		\\& = 2\sin\big(\frac{\theta(m+p)-\theta(m)}{2}\big) + \bigO_\alpha\Big(\frac{A}{|m|^\alpha \phi(N)}\Big). \numberthis\label{eq:linear_approx_link_theta_p/m}
	\end{align*}
	}Using this, we claim that the quantity $\theta(m+p)-\theta(m)$ which belongs to $\, ]\pi, 2\pi[ \,$ since $(m,p) \in I_{N,A}^{\rm bad}$, has to be close to $2\pi$. Since the image of $\theta$ is $[0,2\pi[ \,$, this will imply that $\theta(m+p)$ has to be close to $2\pi$ while $\theta(p)$ has to be close to $0$. In other words, both grid points $m+p$ and $m$ have to be close to the real positive ray $\RR_+-\{0\}$. Using the concavity of the sinus function on $[0,\frac{\pi}{2}]$, we can derive the following estimate from Equation \eqref{eq:linear_approx_link_theta_p/m} (and using again Equation \eqref{eq:linear_approx_unif_conv_p/m}), as $N\to\infty$,
	\begin{align*}
		\frac{2}{\pi}(2\pi - (\theta(m+p)-\theta(m))) & \leq 2\sin\big( \frac{\theta(m+p)-\theta(m)}{2})
		\\ & = \big|\frac{p}{m}\big| + \bigO_\alpha\Big(\frac{A}{|m|^\alpha \phi(N)}\Big)
		\\ \mbox{thus~~} 2\pi - (\theta(m+p)-\theta(m)) & = \bigO_\alpha\Big(\frac{A}{|m|^\alpha \phi(N)}\Big)\numberthis\label{eq:linear_approx_bdn_theta}
	\end{align*}
	which proves the claim.
	
	As an immediate consequence, the same estimate holds for $2\pi - \theta(m+p)$ and for $\theta(m)$. We choose a constant $C_\alpha>0$ to make the Landau's notation explicit so that $\theta(m) \leq C_\alpha \frac{A}{|m|^\alpha \phi(N)}$, then we set the function $g_N:x\mapsto x \tan\big( C_\alpha \frac{A}{x^\alpha \phi(N)} \big)$. For $N$ large enough so that $C_\alpha \frac{A}{\phi(N)}<\frac{\pi}{2}$, the map $g_N$ is well defined over $[1,+\infty[\,$ and is nondecreasing. Applying Lemma \ref{lem:grid_pts_near_line} with $g_N$ gives us the estimate, as $N\to\infty$,
	\begin{align*}
		& \card \Big\{ m \in (\Lambda-\{0\})\cap D(0,N) \, : \, \theta(m) \leq \frac{C_\alpha A}{|m|^\alpha \phi(N)} \Big\}
		\\ \leq~ & N \frac{4(1+\diam_{\vec{\Lambda}})(N \tan( C_\alpha \frac{A}{N^\alpha \phi(N)}) + \diam_{\vec{\Lambda}})}{\covol_{\vec{\Lambda}}}
		\\ =~ &\bigO_{\alpha,\Lambda} \Big( N(\frac{AN^{1-\alpha}}{\phi(N)}+1) \Big).
	\end{align*}
	Multiplying this bound by the number of lattice points $p$ described in Equation \eqref{eq:linear_approx_card_p_lambda=0} gives us the following estimate for counting these bad indices, as $N\to\infty$,
	$$
	\card (I_{N,A}^{\rm bad}) =\bigO_{\alpha,\Lambda} \Big( N \big(\frac{AN^{1-\alpha}}{\phi(N)} +1 \big)^3 \Big).
	$$
	Thus, the restriction to these bad indices in the error term $\R_{N,N',N''}^{\alpha,\Lambda,+}(f) - \mu_{N,N',N''}^+(f)$ is estimated by, as $N\to\infty$,
	\begin{equation}\label{eq:linear_approx_contrib_bad_indices}
		\bigO_{\alpha,\Lambda} \Big( \frac{\|f\|_\infty N(\frac{AN^{1-\alpha}}{\phi(N)}+1)^3 }{\psi(N)} \Big)
	\end{equation}
	
	We set $I_{N,A}^{\rm good} = I_{N,A}^+ - I_{N,A}^{\rm bad}$. Using the mean value theorem, for all $(m,p) \in I_{N,A}^{\rm good}$, since $\Im(\frac{p}{m}) \geq 0$ by definition of $I_{N,A}^{\rm good}$ and using the uniform estimate \eqref{eq:linear_approx_unif_conv_p/m}, the quantity \eqref{eq:complex_linear_approx_m,p} is bounded by
	\begin{equation}\label{eq:complex_linear_approx_meanvalue}
		\|df\|_\infty \phi(N) |m|^\alpha \, \big|\big(1+\frac{p}{m}\big)^\alpha - 1 - \frac{\alpha p}{m}\big| =\bigO_\alpha\Big(\|df\|_\infty \phi(N) \frac{|p|^2}{|m|^{2-\alpha}}\Big).
	\end{equation}
	It remains to bound from above the sum $S_{N,A} = \sum_{(m,p)\in I_{N,A}^{\rm good}} \frac{|p|^2}{|m|^{2-\alpha}}$. For that matter, we use the estimates \eqref{eq:linear_approx_unif_conv_p/m} (in the form $|p|^2 = \bigO_\alpha \big( \frac{A^2 |m|^{2-2\alpha}}{\phi(N)^2}\big)$) and \eqref{eq:linear_approx_card_p_lambda=0} then we apply again Lemma \ref{lem:sum_powers_grid} (summing over $\Lambda$, with $\beta=-\alpha$ and $x=N$), which gives us an upper bound for the sum $S_{N,A}$ as follows
	\begin{align*}
		S_{N,A} & \leq \sum_{\substack{m \in \Lambda \\ 0<|m|\leq N}} \frac{1}{|m|^{2-\alpha}} \bigO_\alpha \Big( \frac{A^2 |m|^{2-2\alpha}}{\phi(N)^2} \Big) \card \{ p \in \vec{\Lambda}-\{0\} \, : \, (m,p) \in I_{N,A}^{\rm good}\} \\
		& \leq \sum_{\substack{m \in \Lambda \\ 0<|m|\leq N}} \frac{1}{|m|^\alpha} \bigO_\alpha \Big( \frac{A^2}{\phi(N)^2} \Big) \bigO_{\alpha,\Lambda} \Big( \Big(\frac{AN^{1-\alpha}}{\phi(N)} + 1\Big)^2\Big) \\
		& =  \bigO_{\alpha,\Lambda} \Big( \frac{A^2\big(\frac{AN^{1-\alpha}}{\phi(N)} + 1\big)^2}{\phi(N)^2}\Big) \sum_{\substack{m \in \Lambda \\ 0<|m|\leq N}} \frac{1}{|m|^\alpha} \\
		& =\bigO_{\alpha,\Lambda} \Big( \frac{A^2 N^{2-\alpha} \big(\frac{AN^{1-\alpha}}{\phi(N)} + 1\big)^2}{\phi(N)^2}\Big).
	\end{align*}
	This estimate together with the one over $I_{N,A}^{\rm bad}$ given in Equation \eqref{eq:linear_approx_contrib_bad_indices}, and the bound given in Equation \eqref{eq:complex_linear_approx_meanvalue} gives us, as $N\to\infty$,
	{\small
	$$
	\R_{N,N',N''}^{\alpha,\Lambda,+}(f) - \mu_{N,N',N''}^+(f) = \bigO_{\alpha,\Lambda} \Big( \frac{A^2 \|df\|_\infty N^{2-\alpha} \big(\frac{AN^{1-\alpha}}{\phi(N)} + 1 \big)^2}{\psi(N) \phi(N)} + \frac{\|f\|_\infty N(\frac{AN^{1-\alpha}}{\phi(N)}+1)^3 }{\psi(N)} \Big).
	$$
	}Since the renormalization factor is given by the formula $\psi(N)=\big(\frac{N^{2-\alpha}}{\phi(N)}\big)^2$, the latter estimate can be simplified (using the inequality $(a+b)^k \leq 2^k (a^k + b^k)$ for real numbers $a,b >0$ and $k\in\NN$) and finally rewritten, as $N\to\infty$,
	{\footnotesize
	$$
	\R_{N,N',N''}^{\alpha,\Lambda,+}(f) - \mu_{N,N',N''}^+(f) = \bigO_{\alpha,\Lambda} \Big( A^2 \|df\|_\infty \big(\frac{A^2}{N^\alpha\phi(N)} + \frac{\phi(N)}{N^{2-\alpha}}\big) + \|f\|_\infty \big(\frac{A^3}{N^\alpha\phi(N)}+\frac{\phi(N)^2}{N^{3-2\alpha}}\big) \Big).
	$$
	}
\end{proof}

\begin{remark}\label{rk:complex_linear_approx_case_lambda_0}
	{\rm
	If $\lambda=0$, the error term in Lemma \ref{lem:complex_linear_approx} becomes, as $N\to\infty$,
	\begin{align*}
	\R_{N,N',N''}^{\alpha,\Lambda,+}(f) - \mu_{N,N',N''}^+(f) & = \bigO_{\alpha,\Lambda} \Big( \|df\|_\infty \frac{A^4}{N^\alpha\phi(N)}
	+ \|f\|_\infty \frac{A^3}{N^\alpha\phi(N)} \Big)
	\\ & = \bigO_{\alpha,\Lambda} \Big( \frac{A^4 (\|f\|_\infty+\|df\|_\infty)}{N^\alpha \phi(N)} \Big).
	\end{align*}
	}
\end{remark}

\subsection{Riemann sum approximation}\label{ssec:riem_sum_lemma}
The last lemma is a standard Riemann sum approximation. Again, let $\F$ be a closed fundamental parallelogram of $\vec{\Lambda}$ containing $0$ and of diameter $\diam_{\vec{\Lambda}}$. 
\begin{lemma}\label{lem:complex_riemann_approx}
	Let $\delta\in\CC^*$ and $F$ be a finite subset of $\Lambda$. Then, for every function $f \in C^1(\CC)$, we have the inequality
	$$ \Big| |\delta|^2 \covol_{\vec{\Lambda}} \sum_{m \in F} f(m \delta) - \int_{\underset{m\in F}{\bigcup}\delta(m+\F)} f(z)\, dz \Big| \leq \card(F) |\delta|^3 \|df_{\, | \underset{m\in F}{\bigcup}\delta(m+\F)}\|_\infty \diam_{\vec{\Lambda}}.$$
\end{lemma}

\begin{proof}
Notice that, for all $m\in F$, we have $\Leb_\CC(\delta(m+\F)) = \covol_{\delta\vec{\Lambda}} = |\delta|^2\covol_{\vec{\Lambda}}$. A direct application of the mean value inequality for $f$ on the convex sets $\delta(m+\F)$ and then summing over $m\in F$ ends the proof.
\end{proof}
We now have enough tools to prove our effective theorem.


\section{Proof of Theorem \ref{th:effective_cv_complex_correlations}}\label{sec:proof_thm}
We have three different regimes for the scaling factor and the proof will be divided accordingly. Recall that the renormalization is given by the formula $\psi(N)=\big(\frac{N^{2-\alpha}}{\phi(N)}\big)^2$. Let $f \in C_c^1(\CC)$ and choose $A>1$ such that $\supp f \subset D(0,A)$.
	
\subsection[regime +infinite]{Regime $\frac{\phi(N)}{N^{1-\alpha}} \underset{N\to\infty}{\longrightarrow} +\infty$}\label{ssec:regime_infinite}
	
Compared to both other regimes where we get an asymptotic bound for the speed of convergence, this one is particular as we will asymptotically prove the equality $\R_{N,N',N''}^{\alpha,\Lambda,\lvl}(f)=0$ representing a drastic loss of mass at infinity. For that reason, we will not use whole lemmas from Section \ref{sec:lemmas} but only elements of their proof. For $N$ large enough (independently on $N',N''$), we will first prove the equality $\R_{N,N',N''}^{\alpha,\Lambda,+}(f)=0$ (hence $\R_{N,N',N''}^{\alpha,\Lambda,-}(f)=0$ by symmetry), then we will take care of the diagonal terms $(m,p)\in I_N$, that is to say those which verify $\frac{m+p}{m} \in \RR$.
	
Fix $k\in\ZZ$. Recall that the set $I_N^+$ is defined so that, for all indices $(m,p) \in I_N^+$, the formula $(m+p)^{[\alpha,k]}-m^{[\alpha,k]} = m^{[\alpha,k]} ((1+\frac{p}{m})^\alpha - 1)$ holds. Our goal is to prove that, for $N$ large enough independently on $k$, we have the inequality
$$\big| \big( 1+\frac{p}{m} \big)^\alpha - 1 \big| \geq \frac{A}{|m|^\alpha \phi(N)}.$$
Using the notation $I_{N,A}^+$ from the proof of Lemma \ref{lem:complex_linear_approx}, the indices $(m,p)\in I_N^+$ failing to verify the former inequality are in this set $I_{N,A}^+$ by Equation \eqref{eq:complex_bounds_INA}. Thus it is sufficient to prove that, for $N$ large enough, we have $I_{N,A}^+=\emptyset$. For all $(m,p) \in I_{N,A}^+$, we can use the estimate $|p|=\bigO_\alpha(\frac{AN^{1-\alpha}}{\phi(N)})$, that follows from Equation \eqref{eq:linear_approx_unif_conv_p/m}. Thanks to the convergence $\frac{N^{1-\alpha}}{\phi(N)}\to 0$ as $N\to\infty$ and the inequality $|p| \geq \sys_{\vec{\Lambda}}$ for all $p \in \vec{\Lambda}-\{0\}$, we have indeed $I_{N,A}^+=\emptyset$ for $N$ large enough. For such ranks $N$ and for all $N',N''\in\NN$, this immediately gives the equality $\R_{N,N',N''}^{\alpha,\Lambda,+}(f)=0$. With the same condition on the ranks $N$, $N'$ and $N''$, the equality $\R_{N,N',N''}^{\alpha,\Lambda,-}(f)=0$ follows from the symmetry described in Equation \eqref{eq:symmetry}.
	
The same argument, this time using the set of indices $I_{N,A}^=$ defined in the proof of Lemma \ref{lem:negligible_diag_points} and the estimate \eqref{eq:bnd_diag_p}, gives the result over the diagonal terms. After summing over $I_{N,A}^+ \cup I_{N,A}^- \cup I_{N,A}^=$, we have finally proven the equality, for $N$ large enough and for all $N',N'' \in \NN$,
$$
\R_{N,N',N''}^{\alpha,\Lambda,\lvl}(f)=0.
$$

\subsection{Local changes of variables}\label{ssec:change_var}
\subsubsection{Riemann sums argument}\label{sssec:riem_sum_change_var}
In the two other regimes for the scaling factor $\phi$, thanks to the symmetry equation \eqref{eq:symmetry} and to Lemmas \ref{lem:negligible_diag_points} and \ref{lem:complex_linear_approx}, it is sufficient to study the behaviour of the sequence $(\mu_{N,N',N''}^+(f))_{N,N',N''\in\NN}$ defined by the formula that we recall
$$
\mu_{N,N',N''}^+(f) = \frac{1}{(N'+N'')\psi(N)} \sum_{k=-N'}^{N''-1} \sum_{p\in\vec{\Lambda}-\{0\}} \sum_{\substack{m \in \Lambda \\ (m,p)\in I_N^+}} f\Big( \frac{\phi(N) \alpha p}{m^{[1-\alpha,k]}}  \Big),
$$
where $I_N^+=\{ (m,p) \in \Lambda \times ( \vec{\Lambda}-\{0\}) \, : \, 0<|m|,|m+p|\leq N \mbox{ and } \theta(m+p) > \theta(m) \}$.. In order for an index $(m,p)$ to contribute to this sum, it has to verify, as $N \to \infty$,
\begin{equation}\label{eq:bound_p_m}
	\big|\frac{p}{m}\big| \leq \frac{A}{\alpha |m|^\alpha \phi(N)} \; \mbox{ hence } \; |p| \leq \frac{A N^{1-\alpha}}{\alpha \phi(N)}.
\end{equation}
In order to see the measure $\mu_{N,N',N''}$ as a weighted Riemann sum over the lattice $\vec{\Lambda}$, we will use the open angular sector (illustrated in Figure \ref{fig:geom_symm_after_change_var})
$$
C_{p,k} = \{ z \in \CC^* \, : \, \arg(z) \in \theta(p) - (1-\alpha)2\pi\, ]k,k+1[\, +2\pi\ZZ \},
$$ 
the ray $L_{p,k} = \{ z \in \CC^* \, : \, \arg(z) \equiv - \frac{\theta(p)}{1-\alpha}\}$ and the family of change of variables $(h_{p,k})_{p \in \vec{\Lambda}, \, k\in\ZZ}$ defined by
$$
\begin{array}{rcl}
	h_{p,k} : C_{p,k} & \to & \CC^* - L_{p,k} \\
	z & \mapsto & |z|^{-\frac{1}{1-\alpha}} e^{-\frac{i\omega_k}{1-\alpha}} \mbox{ with } \arg(z) \equiv \omega_k \in \theta(p) - (1-\alpha)2\pi\, ]k,k+1[\,.
\end{array}
$$
In other words, these changes of variables are restrictions to $C_{p,k}$ of the maps
\begin{equation}\label{eq:formula_log_change_var}
	h_{p,k}: z \mapsto \exp \Big( -\frac{1}{1-\alpha} ( \log(ze^{i(-\theta(p)+2\pi(k+1)(1-\alpha))}) + i (\theta(p) - 2\pi(k+1)(1-\alpha))) \Big)
\end{equation}
where $\log$ is the nonstandard branch of the logarithm on $\CC-\RR_+$ defined in the beginning of Section \ref{ssec:levels}. Let $p\in\vec{\Lambda}-\{0\}$ and $k\in\ZZ$. The map $h_{p,k}$ is biholomorphic and computing its derivative, using the formula \eqref{eq:formula_log_change_var}, gives us $h_{p,k}' : z \mapsto - \frac{1}{1-\alpha} \frac{h_{p,k}(z)}{z}$, whose modulus is $z \mapsto \frac{1}{1-\alpha} |z|^{-\frac{2-\alpha}{1-\alpha}}$. We set
$$
\omega_{p,k} = \sum_{\substack{m \in \Lambda, \, m\notin \RR_+ \\ (m,p)\in I_N^+}} \Delta_{\frac{\phi(N) \alpha p}{m^{[1-\alpha,k]}}}
$$
allowing us to decompose the measure $\mu_{N,N',N''}^+$ into sums where $k$ and $p$ are fixed, then apply a different change of variables on each part. The condition $m\notin \RR_+$ is introduced so that the points $\frac{\phi(N)\alpha p }{m^{[1-\alpha,k]}}$ all belong to $C_{p,k}$ and not only to its closure. In order to add or remove this condition at will, notice the inequality
\begin{equation}\label{eq:error_mu_N_sum_omega_pk}
	\card(\Lambda \cap \bar{D}(0,N) \cap \RR_+) \leq \frac{N}{\sys_{\vec{\Lambda}}}+1.
\end{equation}
For all $m\in\Lambda$ such that $(m,p) \in I_N^+$ and $m\notin \RR_+$, the change of variable $h_{p,k}$ is designed for the following computation:
\begin{align*}
	h_{p,k} \big( \frac{\phi(N) \alpha p}{m^{[1-\alpha,k]}} \big) & = \big( \frac{\phi(N)\alpha |p|}{|m|^{1-\alpha}} \big)^{-\frac{1}{1-\alpha}} \, h_{p,k} \big( e^{i(\theta(p)-(1-\alpha)(\theta(m)+2\pi k))} \big)
	\\ & = \big( \frac{\phi(N)\alpha |p|}{|m|^{1-\alpha}} \big)^{-\frac{1}{1-\alpha}} \, e^{-\frac{1}{1-\alpha} i(\theta(p) - (1-\alpha)(\theta(m) + 2\pi k))} = (\phi(N)\alpha p)^{-\frac{1}{1-\alpha}} \, m
\end{align*}
where we recall the notation $z^{-\frac{1}{1-\alpha}} = z^{[-\frac{1}{1-\alpha},0]}$. Consequently, we have the formula
\begin{equation}\label{eq:h_pk_omega_pk}
	(h_{p,k})_* \omega_{p,k} = \sum_{\substack{m \in \Lambda, \; m\notin \RR_+ \\ (m,p)\in I_N^+}} \Delta_{m \delta_{N,p}} \mbox{ where } \delta_{N,p} = (\phi(N)\alpha p)^{-\frac{1}{1-\alpha}}.
\end{equation}
Using Equation \eqref{eq:error_mu_N_sum_omega_pk}, the condition $m\notin \RR_+$ in the latter formula can be removed up to an extra error term of order $\bigO_\alpha(\frac{\|f\|_\infty N}{\psi(N)\sys_{\vec{\Lambda}}})$, thus we forget about it until Equation \eqref{eq:finalestimate_change_var}.

In order to compare every measure $(h_{p,k})_* \omega_{p,k}$ with a weighted Riemann sum, we have to establish which part of $\CC$ is occupied by the indices $m$ in its definition. Recall that $I_N^+$ denotes the subset of $\Lambda\times (\vec{\Lambda}-\{0\})$ with conditions $0 < |m|, |m+p| \leq N$ and $\theta(m+p) > \theta (m)$. Putting aside the condition $|m+p|\leq N$ for the moment, we claim that such indices $m$ approximately occupy a half-disk (depending on $p$), namely half of the closed disk $D(0,N)$. Let $B_p$ denote the complex band $[-1,0]p+\RR_+$. More precisely, we claim that, modulo the complex subset $B_p \cap \bar{D}(0,N)$, the set
$$
D_{N,p} = \{ z \in \CC^*-\{-p\} \, : \, |z| \leq N \mbox{ and } \theta(z+p) > \theta(z) \}
$$
is the half-disk centred at the origin, of radius $N$ and with the argument condition $\theta(z) \in \, ]\theta(p)-\pi, \theta(p)[\, + 2\pi\ZZ$. The claim follows from a straightforward study of (the sign of) the function $z \mapsto \theta(z+p)-\theta(z)$ which is continous on $\CC-(\RR_+ \cup (-p+\RR_+) )$, which can be computed explicitly on the circle $C(0,|p|)$ and whose zeros belong to the line $\RR p$. See Figure \ref{fig:approx_half_disk} for a summary of this study. A quantitative comparison between $D_{N,p}$ and the associated half-disk will be stated in Equation \eqref{eq:leb_meas_symm_diff_after_chg_var}. 
\begin{figure}[ht]
	\centering
	\includegraphics[height=6.5cm]{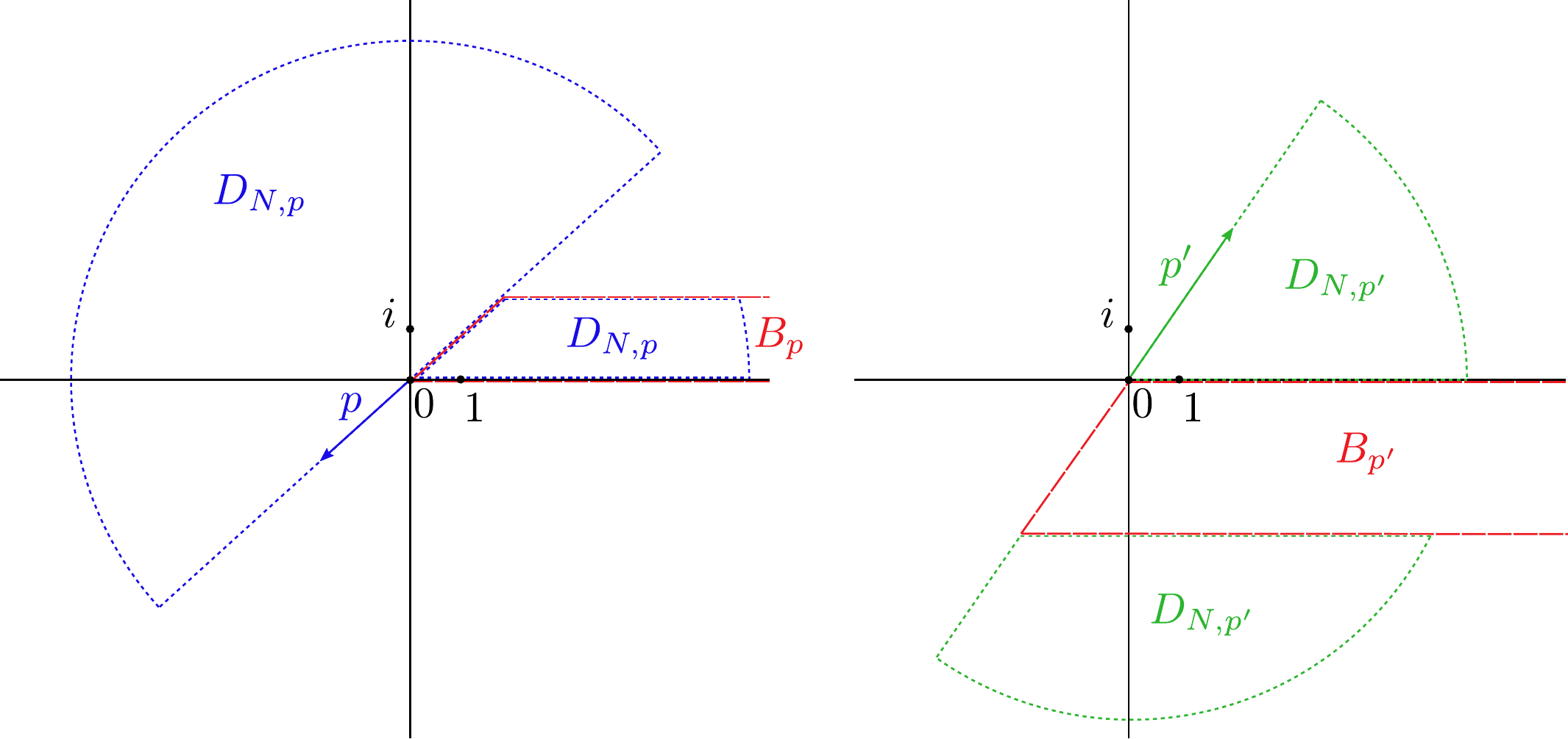}
	\caption{An illustration of a set $D_{N,p}$ with $\theta(p) \geq \pi$, and $D_{N,p'}$ with $\theta(p')<\pi$.}
	\label{fig:approx_half_disk}
\end{figure}

In order to remove the condition $|m+p| \leq N$ in $I_N^+$ and to compute the associated error term, first notice that failing this condition implies that $N-|p| < |m| \leq N$. Using Lemma \ref{lem:sum_powers_grid} twice (summing over $\Lambda$ with $\beta=0$, first with $x=N$ then with $x=N-|p|$), we obtain, as $N\to\infty$ with $N \geq |p|$,
\begin{align*}
\card(\Lambda \cap D_{N,p} - \{m\in\Lambda \, : \, (m,p) \in I_N^+\}) = \, & \frac{\pi(N^2 - (N-|p|)^2)}{\covol_{\vec{\Lambda}}} + \bigO \big(\frac{1+\diam_{\vec{\Lambda}}^2}{\covol_{\vec{\Lambda}}} N\big) 
\\ = \, & \bigO_{\Lambda} \Big( (|p| + 1) N\Big).
\end{align*}
Thanks to the inequality $|p| \leq \frac{AN^{1-\alpha}}{\alpha \phi(N)}$ from Equation \eqref{eq:bound_p_m}, the condition $N \geq |p|$ in the latter estimate is verified for $N$ large enough, uniformly on such indices $p$. Using Lemma \ref{lem:sum_powers_grid} (summing over $\vec{\Lambda}$ with $\beta=0$ and $x= \frac{AN^{1-\alpha}}{\alpha \phi(N)}$), we can replace the condition $(m,p) \in I_N^+$ by $m \in \Lambda \cap D_{N,p}$ in the definition of $\mu_{N,N',N''}(f)$ up to the error term, as $N\to\infty$,
\begin{align*}
	& \mu_{N,N',N''}^+(f) - \frac{1}{(N'+N'')\psi(N)} \sum_{k=-N'}^{N''-1} \sum_{p\in\vec{\Lambda}-\{0\}} \sum_{m\in \Lambda \cap D_{N,p}} f(\frac{\phi(N)\alpha p}{m^{[1-\alpha,k]}})
	\\ = \; & \bigO_{\alpha,\Lambda }\Big( \frac{\|f\|_\infty(\frac{AN^{1-\alpha}}{\phi(N)}+1)(\frac{AN^{1-\alpha}}{\phi(N)})^2 N}{\psi(N)} \Big). \numberthis\label{eq:estime_remove|m+p|lessN}
\end{align*}
This invites us to define the measures
$$
\widetilde{\omega}_{p,k} = \sum_{\substack{m \in \Lambda, \, m\notin \RR_+ \\ m \in \Lambda \cap D_{N,p}}} \Delta_{\frac{\phi(N) \alpha p}{m^{[1-\alpha,k]}}}
$$
and $\widetilde{\mu}_{N,N',N''}^+ = \frac{1}{(N'+N'')\psi(N)}\sum_{k=-N'}^{N''-1} \sum_{p\in\vec{\Lambda}-\{0\}} \widetilde{\omega}_{p,k}$. Using Equations \eqref{eq:error_mu_N_sum_omega_pk} and \eqref{eq:estime_remove|m+p|lessN}, we obtain an error term, as $N\to\infty$,
{\footnotesize
\begin{equation}\label{eq:error_mu_tilde_N_sum_omega_tilde_pk}
	\widetilde{\mu}_{N,N',N''}^+(f) - \mu_{N,N',N''}^+(f) = \bigO_{\alpha,\Lambda} \Big( \frac{\|f\|_\infty(\frac{AN^{1-\alpha}}{\phi(N)}+1)(\frac{AN^{1-\alpha}}{\phi(N)})^2 N}{\psi(N)} + \frac{\|f\|_\infty N}{\psi(N)} \Big).
\end{equation}
}Let $\F$ be a fundamental domain of $\vec{\Lambda}$ containing $0$ and of diameter $\diam_{\vec{\Lambda}}$. An approximation of $D_{N,p}$ is given by $\widetilde{D}_{N,p}=\bigcup_{m\in\Lambda\cap D_{N,p}}(m+\F)$. We apply Lemma \ref{lem:complex_riemann_approx} (on the $C^1$ function $f_{p,k} = f \circ h_{p,k}^{-1}$, with $\delta=\delta_{N,p}$ and $F=\Lambda \cap D_{N,p}$), then we use Lemma \ref{lem:sum_powers_grid} (summing over $\Lambda$ with $\beta=0$ and $x=N$ since we have the inclusion $D_{N,p}\subset \bar{D}(0,N)$). This grants us the estimate, as $N\to\infty$,
{\small
\begin{align*}
\Big| (h_{p,k})_* \widetilde{\omega}_{p,k}(f_{p,k}) - \frac{\int_{\delta_{N,p}\widetilde{D}_{N,p}} f_{p,k}(z) \, dz}{|\delta_{N,p}|^2 \covol_{\vec{\Lambda}}} \Big| & \leq \frac{\diam_{\vec{\Lambda}}}{\covol_{\vec{\Lambda}}}  \|df_{p,k}\,_{| \delta_{N,p} \widetilde{D}_{N,p}}\|_\infty \, |\delta_{N,p}| \card(\Lambda \cap D_{N,p})
\\ & = \bigO_{\alpha,\Lambda} \Big( \frac{\|df_{p,k}\,_{| \delta_{N,p} \widetilde{D}_{N,p}}\|_\infty N^2 }{(\phi(N) |p|)^{\frac{1}{1-\alpha}}} \Big). \numberthis\label{eq:estimate_riem_sum_1}
\end{align*}
}The set $\widetilde{D}_{N,p}$ is "approximately" $D_{N,p}$, and is "approximately" a half-disk as we shall now see. Let us use the notation, for all $z_0\in\CC$, $r>0$, $\omega \in \RR$, 
$$H(z_0, r, \omega) = \{z\in\CC \, : \, |z-z_0| \leq r \mbox{ and } \arg(z-z_0) \in \, ]\omega-\pi,\omega[ \, + 2\pi\ZZ \}$$
which is a half-disk centred at $z_0$, of radius $r>0$, with an argument (relative to its centre) determined by $\omega$ (more precisely by its image in $\RR/2\pi\ZZ$). We want to compare the complex subset $\widetilde{D}_{N,p}$ with the half-disk $H(0,N,\theta(p))$. Let $u$ be the complex number verifying $\arg(u)=\arg(p)+\frac{\pi}{2}$ and $|u|=\diam_{\vec{\Lambda}}$. Let $\widetilde{B}_{p,N}$ denote the $\diam_{\vec{\Lambda}}$-neighbourhood of the band $B_p \cap \bar{D}(0,N)$. Using the triangle inequality, modulo the set $\widetilde{B}_{p,N}$, we have the following inclusions
$$
\widetilde{D}_{N,p} \subset H(u,N+2\diam_{\vec{\Lambda}},\theta(p)) \; \mbox{ and } \; H(-2u, N-3\diam_{\vec{\Lambda}}, \theta(p)) \subset \widetilde{D}_{N,p}.
$$
(We don't necessarily have $H(-u, N-2\diam_{\vec{\Lambda}}, \theta(p)) \subset \widetilde{D}_{N,p} \cup \widetilde{B}_{p,N}$ in the case where $\Lambda$ contains $0$, since $0$ never belongs to $\Lambda \cap D_{N,p}$ which is the set of indices we defined $\widetilde{D}_{N,p}$ with). Thus, the symmetric difference that is of interest here verifies, modulo $\widetilde{B}_{p,N}$,
\begin{align*}
\widetilde{D}_{N,p} \Delta H(0,N,\theta(p)) & = (\widetilde{D}_{N,p} \cup H(0,N,\theta(p)))-(\widetilde{D}_{N,p} \cap H(0,N,\theta(p))) 
\\ & \subset H(u,N+2\diam_{\vec{\Lambda}},\theta(p)) - H(-2u, N-3\diam_{\vec{\Lambda}}, \theta(p)).
\end{align*}
Since the set $\widetilde{B}_{p,N}$ has Lebesgue measure bounded by $\bigO((|p|+\diam_{\vec{\Lambda}})N)$, the latter inclusion modulo $\widetilde{B}_{p,N}$ gives the estimate, as $N\to\infty$ (with $N \geq 3\diam_{\vec{\Lambda}}$ and independently on $p \in \vec{\Lambda}-\{0\}$),
{\small
\begin{align*}
	\Leb_\CC(\widetilde{D}_{N,p} \Delta H(0,N,\theta(p))) & \leq \frac{\pi}{2}((N+2\diam_{\vec{\Lambda}})^2 - (N-3\diam_{\vec{\Lambda}})^2) + \bigO((|p|+\diam_{\vec{\Lambda}})N)
	\\ & = \bigO(\diam_{\vec{\Lambda}} N) + \bigO((|p|+\diam_{\vec{\Lambda}})N) = \bigO_\Lambda((|p|+1)N). \numberthis\label{eq:leb_meas_symm_diff_after_chg_var}
\end{align*}
}

Let $R_N=\frac{N+\diam_{\vec{\Lambda}}}{(\phi(N)\alpha|p|)^{\frac{1}{1-\alpha}}}$. From the estimates \eqref{eq:estimate_riem_sum_1} and \eqref{eq:leb_meas_symm_diff_after_chg_var}, we derive, as $N\to\infty$,
{\footnotesize
\begin{align*}
	& \Big| (h_{p,k})_* \widetilde{\omega}_{p,k}(f_{p,k}) - \frac{\int_{\delta_{N,p}H(0,N,\theta(p))} f_{p,k}(z) \, dz}{|\delta_{N,p}|^2 \covol_{\vec{\Lambda}}} \Big|
	\\ = \, &\bigO_{\alpha, \Lambda} \Big( \frac{\|df_{p,k}\,_{| \delta_{N,p} \widetilde{D}_{N,p}}\|_\infty N^2 }{(\phi(N) |p|)^{\frac{1}{1-\alpha}}} \Big) + \frac{|\int_{\delta_{N,p}\widetilde{D}_{N,p}} f_{p,k}(z) \, dz - \int_{\delta_{N,p}H(0,N,\theta(p))} f_{p,k}(z) \, dz|}{\covol_{\vec{\Lambda}} |\delta_{N,p}|^2}
	\\ = \, & \bigO_{\alpha, \Lambda} \Big( \frac{\|df_{p,k}\,_{| \delta_{N,p} \widetilde{D}_{N,p}}\|_\infty N^2 }{(\phi(N) |p|)^{\frac{1}{1-\alpha}}} \Big)
	\\ & + \frac{\|f_{p,k} \, _{| \delta_{N,p} (\widetilde{D}_{N,p} \cup H(0,N,\theta(p)))}\|_\infty \Leb_\CC(\delta_{N,p} \widetilde{D}_{N,p} \Delta \delta_{N,p} H(0,N,\theta(p)))}{\covol_{\vec{\Lambda}} |\delta_{N,p}|^2}
	\\ = \, & \bigO_{\alpha, \Lambda} \Big( \frac{
		\|df_{p,k}\,_{| D(0,R_N)} \|_\infty N^2}{(\phi(N) |p|)^{\frac{1}{1-\alpha}}} +
		\| f_{p,k}\,_{| D(0,R_N)} \|_\infty (|p|+1)N \Big). \numberthis\label{eq:estimate_riem_sum_2}
\end{align*}
}We set
$$HC_{p,k} = h_{p,k}^{-1}\big(\delta_{N,p}H(0,N,\theta(p)) - L_{p,k}\big) = h_{p,k}^{-1}\big(H(0,|\delta_{N,p}|N, -\frac{\alpha}{1-\alpha} \theta(p))-L_{p,k}\big)$$
(where we used the equality $\arg(\delta_{N,p})+\theta(p) \equiv -\frac{1}{1-\alpha}\theta(p) + \theta(p) = -\frac{\alpha}{1-\alpha}\theta(p)$ for the right-hand equality). We will geometrically describe $HC_{p,k}$ in Section \ref{sssec:geom_desc_symm}, and see that this set is the intersection of an angular sector (which turns out to be half of $C_{p,k}$) and the complementary set $\CC-\bar{D}(0,|\delta_{N,p}N|^{-(1-\alpha)})=\CC-\bar{D}(0,\frac{\alpha |p| \phi(N)}{N^{1-\alpha}})$. Recall the formula $f_{p,k} = f \circ h_{p,k}^{-1}$ and that the modulus of $h_{p,k}'$ is $z \mapsto \frac{1}{1-\alpha} |z|^{-\frac{2-\alpha}{1-\alpha}}$. Hence the Jacobian of $h_{p,k}$ is $z\mapsto \frac{1}{(1-\alpha)^2} |z|^{-\frac{4-2\alpha}{1-\alpha}}$. We define
\begin{equation}\label{eq:def_nu_{p,k}}
	\nu_{p,k}^+(f) = \frac{1}{(1-\alpha)^2|\delta_{N,p}|^2 \covol_{\vec{\Lambda}}} \int_{HC_{p,k}} f(z)|z|^{-\frac{4-2\alpha}{1-\alpha}} \,dz,
\end{equation}
and
$$\nu_{N,N',N''}^+ = \frac{1}{(N'+N'')\psi(N)} \sum_{k=-N'}^{N''-1} \sum_{p\in\vec{\Lambda}-\{0\}}\nu_{p,k}^+.$$
Thanks to the inclusion $\supp{f} \subset D(0,A)$ and the formula $|h_{p,k}^{-1}| : z \mapsto |z|^{-(1-\alpha)}$, we have the inequalities, for all $N\in\NN$,
$$\|f_{p,k} \, _{|D(0,R_N)}\|_\infty \leq \|f\|_\infty \, \mbox{ and } \, \|df_{p,k}\,_{| D(0,R_N)} \|_\infty \leq (1-\alpha)A^{(1-\alpha)(2-\alpha)} \|df\|_\infty \leq A^2 \|df\|_\infty.$$
Let $\e_\alpha$ be the function $1+|\log|$ if $\alpha = \frac{1}{2}$, and the constant function $1$ otherwise. Combining Equations \eqref{eq:estimate_riem_sum_2} and \eqref{eq:error_mu_N_sum_omega_pk} (to remove the condition $m\notin\RR_+$ in the definition of $\widetilde{\omega}_{p,k}$), applying the change of variable formula, and using Lemma \ref{lem:sum_powers_grid} (summing over $\vec{\Lambda}$ with $\beta=0$ and $x=\frac{AN^{1-\alpha}}{\alpha \phi(N)}$ thanks to Equation \eqref{eq:bound_p_m}), we compute the estimate, as $N\to\infty$,
{\small
\begin{align*}
	& \widetilde{\mu}_{N,N',N''}^+(f) + \bigO_\alpha \big( \frac{\|f\|_\infty N}{\psi(N)\sys_{\vec{\Lambda}}} \big) - \nu_{N,N',N''}^+(f)
	\\ = \, & \sum_{\substack{p\in\vec{\Lambda}-\{0\}\\ |p| \leq \frac{AN^{1-\alpha}}{\alpha \phi(N)}}}
		\bigO_{\alpha, \Lambda}\Big( 
		\frac{\| f_{p,k}\,_{| D(0,R_N)} \|_\infty (|p|+1)N}{\psi(N)} +
		\frac{\|df_{p,k}\,_{| D(0,R_N)} \|_\infty N^2}{\psi(N) (\phi(N) |p|)^{\frac{1}{1-\alpha}}}
		\Big)
	\\ = \; & \bigO_{\alpha, \Lambda} \Big(
		\frac{\| f \|_\infty (\frac{AN^{1-\alpha}}{\phi(N)}+1)N(\frac{AN^{1-\alpha}}{\phi(N)})^2}{\psi(N)}
		+ \frac{ A^2 \|df\|_\infty N^2}{\psi(N) \phi(N)^{\frac{1}{1-\alpha}}} \sum_{\substack{p\in\vec{\Lambda}-\{0\}\\ |p| \leq \frac{AN^{1-\alpha}}{\alpha \phi(N)}}} |p|^{-\frac{1}{1-\alpha}}
		\Big).
\end{align*}
}Together with Equation \eqref{eq:error_mu_tilde_N_sum_omega_tilde_pk}, we finally obtain the estimate, as $N\to\infty$,
\begin{align*}
	\mu_{N,N',N''}^+(f) - \nu_{N,N',N''}^+(f) \numberthis\label{eq:finalestimate_change_var}
	= & \bigO_{\alpha, \Lambda} \Big( \frac{\|f\|_\infty N}{\psi(N)}
	\big(1 +\big(\frac{AN^{1-\alpha}}{\phi(N)}\big)^2 \big)
	\\ & \hspace{1cm} + \|f\|_\infty\big(1+\frac{AN^{1-\alpha}}{\phi(N)}\big)\big(\frac{AN^{1-\alpha}}{\phi(N)}\big)^2 \frac{N}{\psi(N)}
	\\ & \hspace{1cm} + \frac{ A^2 \|df\|_\infty N^2}{\psi(N) \phi(N)^{\frac{1}{1-\alpha}}} \sum_{\substack{p\in\vec{\Lambda}-\{0\}\\ |p| \leq \frac{AN^{1-\alpha}}{\alpha \phi(N)}}} |p|^{-\frac{1}{1-\alpha}}
	\Big).
\end{align*}

\subsubsection{Geometric description by symmetry}\label{sssec:geom_desc_symm}
Set $\nu_{N,N',N''} = \nu_{N,N',N''}^+ + (z \mapsto -z)_*\nu_{N,N',N''}^+$. Using the symmetry argument \eqref{eq:symmetry}, we will be able to compare $\R_{N,N',N''}^{\alpha,\Lambda,\lvl}$ to $\nu_{N,N',N''}$. This section aims at describing the measure $\nu_{N,N',N''}$.

\begin{lemma}\label{lem:symmetry_angular_sector}
For all $k\in\ZZ$ and all $p\in \vec{\Lambda}-\{0\}$, up to a complex subset of Lebesgue measure $0$, we have the disjoint union
$$
HC_{p,k} \cup (-HC_{-p,k}) = C_{p,k} \cap (\CC - D(0,\frac{\alpha |p| \phi(N)}{N^{1-\alpha}})).
$$
\end{lemma}
\begin{figure}[ht]
	\centering
	\includegraphics[height=7.cm]{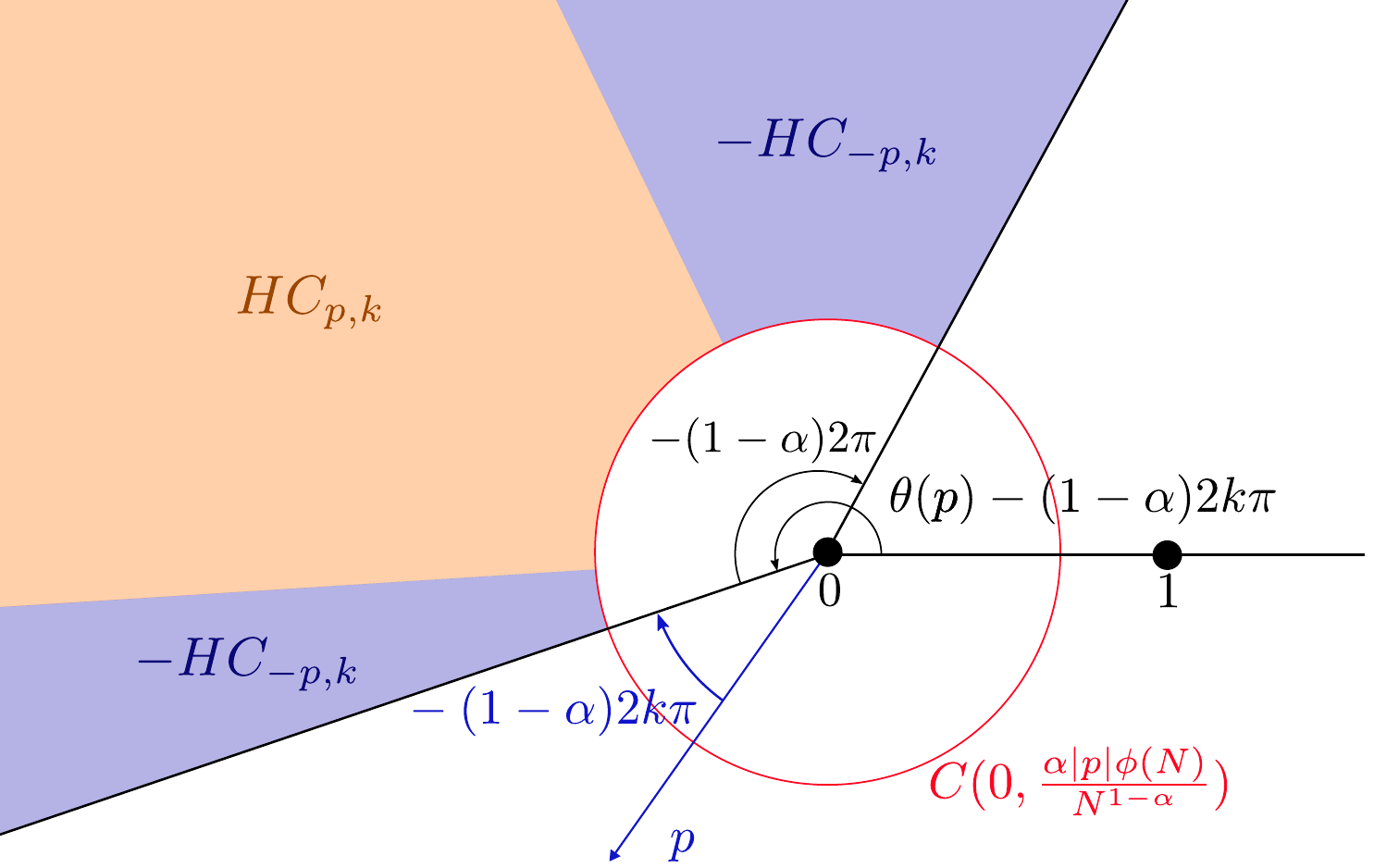}
	\caption{Illustration of Lemma \ref{lem:symmetry_angular_sector}.}
	\label{fig:geom_symm_after_change_var}
\end{figure}
\begin{proof}
	 To prove this, fix $k$ and $p$ as such. We begin by noticing that $HC_{p,k}$ and $-HC_{-p,k}$ are indeed subsets of $C_{p,k}$ (by definition for $HC_{p,k}$, and thanks to the inclusion $-HC_{-p,k} \subset - C_{-p,k} = C_{p,k}$). Furthermore, they are subsets of $\CC - D(0,\frac{\alpha |p| \phi(N)}{N^{1-\alpha}})$ since the changes of variable $h_{p,k}^{-1}$ and $h_{-p,k}^{-1}$ have the same modulus $z \mapsto |z|^{-(1-\alpha)}$. More precisely, the definition of $h_{p,k}$ grants the formula, for all $z\in\CC-L_{p,k}$,
	 $$
	 h_{p,k}^{-1}(z) = |z|^{-(1-\alpha)} e^{-i(1-\alpha)\omega_k} \; \mbox{ where } \; \arg(z) \equiv \omega_k \in -\frac{\theta(p)}{1-\alpha} + 2\pi \, ]k,k+1[
	 $$
	or, in other words, for all $r>0$ and $\omega \in \RR - (-\theta(p) +2\pi\ZZ)$,
	{\small
	\begin{equation}\label{eq:formula_hpk_polar}
		h_{p,k}^{-1}(re^{i(-\frac{\alpha \theta(p)}{1-\alpha} + \omega)}) = r^{-(1-\alpha)} \exp\Big(\theta(p) - (1-\alpha)\big(\theta(p)+\omega+2\pi\big(k-\Big\lfloor\frac{\theta(p)+\omega}{2\pi}\Big\rfloor\big)\big)\Big).
	\end{equation}
	}Since $h_{p,k}^{-1}$ (and similarly $h_{-p,k})$) acts separately on each variable in polar coordinates, it remains to describe $HC_{p,k}$ (resp.~$-HC_{-p,k}$) in terms of arguments, which reduces to the description of the set $h_{p,k}^{-1}(\SSS^1-\{e^{-\frac{i\theta(p)}{1-\alpha}}\})$ (resp.~$-h_{-p,k}^{-1}(\SSS^1-\{e^{-\frac{i\theta(-p)}{1-\alpha}}\})$). Using the formula \eqref{eq:formula_hpk_polar} and doing separately the cases $\theta(p) \leq \pi$ and $\theta(p) > \pi$, we find that
	{\small
	\begin{align*}
		& h_{p,k}^{-1}(\SSS^1-\{e^{-\frac{i\theta(p)}{1-\alpha}}\})
		\\ = \, & \left\{
		\begin{array}{ll}
			\{e^{i\omega} \, : \, \omega \in \theta(p) - 2\pi(1-\alpha) (\, ]k,k+\frac{\theta(p)}{2\pi}[ \, \cup \, ]k+\frac{1}{2}+\frac{\theta(p)}{2\pi}, k+1 [ \,) \} & \mbox{ if } \theta(p) \leq \pi,
			\\ \{e^{i\omega} \, : \, \omega \in \theta(p) - 2\pi(1-\alpha) \, ]k-\frac{1}{2}+\frac{\theta(p)}{2\pi}, k+\frac{\theta(p)}{2\pi} [\,\} & \mbox{ if } \theta(p) > \pi,
		\end{array} \right.
	\end{align*}
	}which is half of the circle arc $\SSS^1 \cap C_{p,k}$. Similarly, up to a finite number of points (namely the three points in $\exp(\alpha \theta(p) -2\pi k(1-\alpha)+\{-1,0,1\})$), the complex subset $-h_{-p,k}^{-1}(\SSS^1-\{e^{-\frac{\theta(-p)}{1-\alpha}}\})$ can be proven equal to half of a circle arc, namely the complement of $h_{p,k}^{-1}(\SSS^1-\{e^{-\frac{i\theta(p)}{1-\alpha}}\})$ in $\SSS^1\cap C_{p,k}$. This concludes the proof of the lemma.
\end{proof}
Thanks to the union described in Lemma \ref{lem:symmetry_angular_sector}, we derive the following formula
$$\nu_{N,N',N''} = \frac{1}{(N'+N'')\psi(N)} \sum_{k=-N'}^{N''-1} \sum_{p\in\vec{\Lambda}-\{0\}} \nu_{p,k}$$
where we set $\nu_{p,k} = \nu_{p,k}^+ + (z\mapsto -z)_* \nu_{-p,k}^+$, that is to say $\nu_{p,k}$ is a measure absolutely continuous with respect to $\Leb_\CC$ with density given by, for all $z\in\CC$,
\begin{equation}\label{eq:def_theta_N_case_0}
	g_{p,k}(z) = \frac{|z|^{-\frac{4-2\alpha}{1-\alpha}}\11_{C_{p,k}\cap(\CC - D(0,\frac{\alpha |p| \phi(N)}{N^{1-\alpha}}))}(z) (\phi(N)\alpha |p|)^{\frac{2}{1-\alpha}}}{(1-\alpha)^2 \covol_{\vec{\Lambda}}}.
\end{equation}
We use the notation $C_p = C_{p,0} \cap (\CC-D(0,\frac{\alpha |p| \phi(N)}{N^{1-\alpha}}))$. We notice that the sector $C_{p,k}\cap(\CC - D(0,\frac{\alpha |p| \phi(N)}{N^{1-\alpha}}))$ is obtained by a rotation of $C_p$ as $e^{-i2\pi k(1-\alpha)} C_p = e^{i2\pi k\alpha}C_p$. We can then describe $\nu_{N,N',N''}$ by the following formula
{\footnotesize
$$
\nu_{N,N',N''}(f) = \frac{(\alpha \phi(N))^{\frac{2}{1-\alpha}}}{(1-\alpha)^2 \covol_{\vec{\Lambda}} (N'+N'')\psi(N)} \sum_{k=-N'}^{N''-1} \sum_{p\in\vec{\Lambda}-\{0\}} |p|^{\frac{2}{1-\alpha}} \int_{e^{i2\pi k\alpha}C_p} f(z)|z|^{-\frac{4-2\alpha}{1-\alpha}} \, dz,
$$
}where the sum over $p\in \vec{\Lambda}-\{0\}$ is finite since $e^{i2\pi k\alpha}C_p \subset \CC - \supp{f}$ if $|p| >\frac{AN^{1-\alpha}}{\alpha\phi(N)}$. When $N',N'' \to \infty$, we will average over $k$ the above integrals on $e^{i2\pi k\alpha}C_p$, which will allow us to replace them by one integral over $\CC-D(0,\frac{\alpha |p| \phi(N)}{N^{1-\alpha}})$. For that purpose, we separate the cases $\alpha \in \QQ \cap \, ]0,1[ \,$ and $\alpha \in (\RR-\QQ) \cap \, ]0,1[\,$. Since the averaging over $k \in \{-N', \ldots, N''-1\}$ and the one over $p\in\vec{\Lambda}-\{0\}$ are geometrically uncorrelated, both averaging processes seem to be necessary in order to obtain a rotation-invariant limit. Imposing a small value of $N'+N''$ empirically leads to rotation discrepancy near the origin, as shown in Figure \ref{fig:correlpairs_nonrotinv_exotic_alpha2342} (where $N'=0$ and $N''=1$).

\noindent The measure we will obtain after this averaging process is given by the formula
\begin{equation}\label{eq:def_nu_N}
	\nu_N(f) = \frac{(\alpha \phi(N))^{\frac{2}{1-\alpha}}}{(1-\alpha) \covol_{\vec{\Lambda}} \psi(N)} \sum_{p\in\vec{\Lambda}-\{0\}} |p|^{\frac{2}{1-\alpha}} \int_{\CC-D(0,\frac{\alpha|p|\phi(N)}{N^{1-\alpha}})} f(z)|z|^{-\frac{4-2\alpha}{1-\alpha}} \, dz.
\end{equation}

\subsubsection{Averaging: the rational case}\label{sssec:average_rat}
In this section, we assume that $\alpha \in \QQ \cap \, ]0,1[ \,$ and we write $\alpha=\frac{a}{b}$ where $a$ and $b$ are coprime positive natural numbers. We recall that the angle of the restricted open sectors $C_p$ is $2\pi (1-\alpha) =  2\pi \frac{b-a}{b}$. Thus, outside of the union of $b$ rays from the origin (which is a set of Lebesgue measure $0$), we have the covering formula, for all $k_0 \in\ZZ$ and all $p\in\vec{\Lambda} - \{0\}$, 
\begin{equation}\label{eq:circle_cover_rot_rat}
\sum_{k=k_0}^{k_0+b-1} \11_{e^{i2\pi k\alpha} C_p} = (b-a) \11_{\CC-D(0,\frac{\alpha|p|\phi(N)}{N^{1-\alpha}})}.
\end{equation}
Hence, we can rewrite $\nu_{N,N',N''}(f)$ by regrouping groups of $b$ consecutive integrals, which gives
{\footnotesize
\begin{align*}
	& \nu_{N,N',N''}(f) = \frac{(\alpha \phi(N))^{\frac{2}{1-\alpha}} (b-a) \big\lfloor \frac{N'+N''}{b} \big\rfloor}{(1-\alpha)^2 \covol_{\vec{\Lambda}} (N'+N'')\psi(N)} \sum_{p\in\vec{\Lambda}-\{0\}} |p|^{\frac{2}{1-\alpha}} \int_{\CC-D(0,\frac{\alpha|p|\phi(N)}{N^{1-\alpha}})} f(z)|z|^{-\frac{4-2\alpha}{1-\alpha}} \, dz
	\\ & + \frac{(\alpha \phi(N))^{\frac{2}{1-\alpha}}}{(1-\alpha)^2 \covol_{\vec{\Lambda}} (N'+N'')\psi(N)} \sum_{k=-N'+\lfloor \frac{N'+N''}{b} \rfloor}^{N''-1}\; \sum_{p\in\vec{\Lambda}-\{0\}} |p|^{\frac{2}{1-\alpha}} \int_{e^{i2\pi k\alpha} C_p} f(z)|z|^{-\frac{4-2\alpha}{1-\alpha}} \, dz.
\end{align*}
}We recall the formula $\psi(N) = (\frac{N^{2-\alpha}}{\phi(N)})^2$. Using polar coordinate, we can bound from above the latter integrals as follows, for all $k \in \ZZ$ and $p \in \vec{\Lambda}-\{0\}$,
{\small
$$
\int_{e^{i2\pi k\alpha} C_p} f(z)|z|^{-\frac{4-2\alpha}{1-\alpha}} \, dz \leq 2\pi \|f\|_\infty \int_{r > \frac{\alpha |p| \phi(N)}{N^{1-\alpha}}} r^{-\frac{3-\alpha}{1-\alpha}}\,dr = (1-\alpha) \pi \|f\|_\infty \big( \frac{\alpha |p| \phi(N)}{N^{1-\alpha}} \big)^{-\frac{2}{1-\alpha}}.
$$
}Using the inequality $\big| \frac{\lfloor \frac{N'+N''}{b} \rfloor}{N'+N''}- \frac{1}{b}\big| \leq \frac{1}{N'+N''}$ and Lemma \ref{lem:sum_powers_grid} (summing over $\vec{\Lambda}$ with $\beta = 0$ and $x=\frac{AN^{1-\alpha}}{\alpha \phi(N)}$), we get the estimate, as $N\to\infty$,
\begin{equation}\label{eq:estimate_nu_average_rat}
	\nu_{N,N',N''}(f) = \nu_N(f) + \bigO_{\alpha, \Lambda} \Big( \frac{A^2 \|f\|_\infty \11_{\ZZ-b\ZZ}(N'+N'')}{N'+N''} \Big).
\end{equation}

\subsubsection{Averaging: the irrational case}\label{sssec:average_irrat}
In this section, we assume that $\alpha \in (\RR-\QQ) \cap \, ]0,1[ \,$. As we take successive rotations by an angle $2\pi\alpha $ (or equivalently, an angle $-2\pi (1-\alpha)$) of the (restricted) angular sector $C_p$, there is no possibility of having a periodic covering formula such as Equation \eqref{eq:circle_cover_rot_rat}. However, since the angle of $C_p$ is also $2\pi (1-\alpha)$, we can still geometrically understand the error in such a covering.
Let $C_{p,N',N''}$ denote the complex subset
{\footnotesize
$$
\big\{z \in \CC \, : \, |z| \geq \frac{\alpha|p|\phi(N)}{N^{1-\alpha}} \mbox{ and } \arg(z) \cap \big(\theta(p) + 2\pi(1-\alpha) \, \big] \frac{\lfloor (1-\alpha)(N'+N'') \rfloor}{1-\alpha} - N'', \, N' \big[\,\big) \neq \emptyset\big\}.
$$
}In other words, $C_{p,N',N''}$ is the restriction to $\CC-D(0,\frac{\alpha|p|\phi(N)}{N^{1-\alpha}})$ of the angular sector between arguments $\theta(p)-(1-\alpha)N''+2\pi\ZZ$ and $\theta(p)+(1-\alpha)N'+2\pi\ZZ$ (with direct trigonometric orientation). Then, outside of the union of $2(N'+N'')$ rays from the origin (which is a set of Lebesgue measure $0$), we have the formula, for all $p\in\vec{\Lambda}-\{0\}$,
$$
\sum_{k=-N'}^{N''-1} \11_{e^{i2\pi k\alpha}C_p} = \lfloor (1-\alpha) (N'+N'') \rfloor \11_{\CC-D(0,\frac{\alpha |p| \phi(N)}{N^{1-\alpha}})} + \11_{C_{p,N',N''}}.
$$
With computations analogous to the ones in Section \ref{sssec:average_rat}, we find a similar error term, namely as $N\to\infty$,
\begin{equation}\label{eq:estimate_nu_average_irrat_2}
	\nu_{N,N',N''}(f) = \nu_N(f) + \bigO_{\alpha, \Lambda} \Big( \frac{A^2 \|f\|_\infty}{N'+N''} \Big).
\end{equation}

\subsection[regime 0]{Regime $\frac{\phi(N)}{N^{1-\alpha}} \underset{N\to\infty}{\longrightarrow} 0$}\label{ssec:regime0}
Using the formula $\psi(N) = (\frac{N^{2-\alpha}}{\phi(N)})^2$ and Lemma \ref{lem:sum_powers_grid} (summing over $\vec{\Lambda}$ with $\beta=-\frac{1}{1-\alpha}$ and $x=\frac{AN^{1-\alpha}}{\alpha \phi(N)}$), the third line in the estimate \eqref{eq:finalestimate_change_var} can be bounded, as $N\to\infty$,
$$
\frac{N^2}{\psi(N) \phi(N)^{\frac{1}{1-\alpha}}} \sum_{\substack{p\in\vec{\Lambda}-\{0\}\\ |p| \leq \frac{AN^{1-\alpha}}{\alpha \phi(N)}}} |p|^{-\frac{1}{1-\alpha}} =
\left\{
\begin{array}{cl}
	A^\frac{1-2\alpha}{1-\alpha} \bigO_{\alpha, \Lambda}(\frac{1}{N}) & \mbox{ if } \alpha < \frac{1}{2},
	\\[3mm] \bigO_{\alpha, \Lambda}(\frac{\log(\frac{\sqrt{N}}{\phi(N)})}{N}) & \mbox{ if } \alpha = \frac{1}{2},
	\\[3mm] \bigO_{\alpha, \Lambda}((N^{2(1-\alpha)} \phi(N)^\frac{2\alpha-1}{1-\alpha})^{-1}) & \mbox{ if } \alpha > \frac{1}{2},
\end{array}
\right.
$$
Since $\frac{1}{N}$, $\frac{\log(\frac{\sqrt{N}}{\phi(N)})}{N}$ and even $(N^{2(1-\alpha)} \phi(N)^\frac{2\alpha-1}{1-\alpha})^{-1}$ are negligible with respect to $\frac{1}{N^\alpha \phi(N)}$, the estimate \eqref{eq:finalestimate_change_var} can be rewritten by removing the first term of its right-hand side and by combining the second and third terms, so that the estimate holds for $N$ large enough independently on $\|df\|_\infty$ (as required in our definition of $\bigO_{\alpha, \Lambda}$). As $N\to\infty$, we obtain
$$
\mu_{N,N',N''}^+(f) = \nu_{N,N',N''}^+(f) + \bigO_{\alpha, \Lambda} \Big( \frac{A^3 (\|f\|_\infty+\|df\|_\infty)}{N^\alpha \phi(N)} \Big).
$$
Then, using the symmetry described in Equation \eqref{eq:symmetry} together with Lemma \ref{lem:negligible_diag_points} (in which the stated estimate is also negligible when compared to $\frac{1}{N^\alpha \phi(N)}$, see Remark \ref{rk:negligible_diag_points_case_lambda_0}) and Lemma \ref{lem:complex_linear_approx}, (see Remark \ref{rk:complex_linear_approx_case_lambda_0}) we get, as $N\to\infty$, 
\begin{equation}\label{eq:estimate_R_N_nu_N_case_0}
	\R_{N,N',N''}^{\alpha,\Lambda,\lvl}(f) = \nu_{N,N',N''}(f) + \bigO_{\alpha, \Lambda} \Big(\frac{A^4 (\|df\|_\infty + \|f\|_\infty)}{N^\alpha \phi(N)} \Big).
\end{equation}
Thanks to the estimates \eqref{eq:estimate_nu_average_rat} and \eqref{eq:estimate_nu_average_irrat_2}, we can focus on the behaviour of the sequence $(\nu_N(f))_{N\in\NN}$ defined in Equation \eqref{eq:def_nu_N}, where $\nu_N$ is the measure of density
$$
g_N : z \mapsto \frac{(\alpha \phi(N))^{\frac{2}{1-\alpha}} |z|^{-\frac{4-2\alpha}{1-\alpha}}}{(1-\alpha)\covol_{\vec{\Lambda}} \psi(N)} \sum_{\substack{p\in\vec{\Lambda}-\{0\} \\ |p| \leq \frac{|z| N^{1-\alpha}}{\alpha\phi(N)}}} |p|^\frac{2}{1-\alpha}
$$
with respect to the Lebesgue measure of $\CC$ (with $g_N(0)=0$ by continuity). In this regime, using Lemma \ref{lem:sum_powers_grid} (summing over $\vec{\Lambda}-\{0\}$ with $\beta=\frac{2}{1-\alpha}$ and $x=\frac{|z| N^{1-\alpha}}{\alpha\phi(N)}$), we have the pointwise convergence
$$
\forall z \in \CC^*, \, g_N(z) \underset{N\to\infty}{\longrightarrow} \frac{\pi}{\alpha^2 (2-\alpha) \covol_{\vec{\Lambda}}^2} = \rho_{\alpha,\vec{\Lambda},\lambda}(z).
$$
More precisely, Lemma \ref{lem:sum_powers_grid} even grants us the error term, as $N\to \infty$, uniformly for every complex number $z \in \CC - D(0,\frac{\alpha \phi(N)}{N^{1-\alpha}})$,
$$
g_N(z) = \rho_{\alpha,\vec{\Lambda},\lambda}(z) + \bigO_{\alpha, \Lambda}\Big( \frac{\phi(N)}{|z| N^{1-\alpha}} \Big).
$$
For all $N\in\NN$, the function $g_N$ vanishes on the open disk $D(0,\frac{\sys_{\vec{\Lambda}} \alpha \phi(N)}{N^{1-\alpha}})$ hence is bounded from above on $D(0,\frac{\alpha \phi(N)}{N^{1-\alpha}})$ by 
$$\frac{(\alpha \phi(N))^{\frac{2}{1-\alpha}} \sys_{\vec{\Lambda}}^{-\frac{4-2\alpha}{1-\alpha}}}{(1-\alpha)\covol_{\vec{\Lambda}} \psi(N)} \sum_{\substack{p\in\vec{\Lambda}-\{0\} \\ |p| \leq 1}} |p|^\frac{2}{1-\alpha} = \frac{\phi(N)^{\frac{2}{1-\alpha}}}{\psi(N)} C_{\alpha,\Lambda}.$$
Integrating these error terms over $\CC$ and since $\int_{D(0,A)} \frac{1}{|z|}\, dz = 2\pi A$, we obtain the estimate, as $N\to\infty$,
{\small
\begin{align*}
	& |\nu_N(f) - \rho_{\alpha,\vec{\Lambda},\lambda} \Leb_\CC(f)|
	\\ \leq \, & \|f\|_\infty \int_{D(0,\frac{\alpha \phi(N)}{N^{1-\alpha}})} \frac{\pi}{\alpha^2 (2-\alpha) \covol_{\vec{\Lambda}}^2} \, dz + \|f\|_\infty \int_{D(0,\frac{\alpha \phi(N)}{N^{1-\alpha}})} \frac{\phi(N)^{\frac{2}{1-\alpha}}}{\psi(N)} C_{\alpha,\Lambda}  \, dz
	\\ & + \|f\|_\infty \int_{D(0,A) \cap (\CC - D(0,\frac{\alpha \phi(N)}{N^{1-\alpha}}))} \bigO_{\alpha, \Lambda}\Big( \frac{\phi(N)}{|z| N^{1-\alpha}} \Big) \, dz
	\\ = \, & \|f\|_\infty \bigO_{\alpha, \Lambda}\Big(\big(\frac{\phi(N)}{N^{1-\alpha}}\big)^2 \Big)
	+ \|f\|_\infty \bigO_\alpha\Big(  C_{\alpha,\Lambda} \big(\frac{\phi(N)}{N^{1-\alpha}}\big)^{\frac{4-2\alpha}{1-\alpha}} \Big) + \|f\|_\infty \bigO_{\alpha, \Lambda} \Big( \frac{A\phi(N)}{N^{1-\alpha}} \Big) 
	\\ = \, & \bigO_{\alpha, \Lambda} \Big( \frac{A \|f\|_\infty \phi(N)}{N^{1-\alpha}} \Big). \numberthis\label{eq:error_nu_N_rho_case_0}
\end{align*}
}Combining Equations \eqref{eq:error_nu_N_rho_case_0}, \eqref{eq:estimate_R_N_nu_N_case_0}, \eqref{eq:estimate_nu_average_irrat_2} and \eqref{eq:estimate_nu_average_rat}, we have proven Theorem \ref{th:effective_cv_complex_correlations} in the case $\lambda=0$.

\subsection[regime lambda]{Regime $\frac{\phi(N)}{N^{1-\alpha}} \underset{N\to\infty}{\longrightarrow} \lambda \in \, ]0,+\infty[\,$}\label{ssec:regimelambda}
By using the inequality $|p| \geq \sys_{\vec{\Lambda}}$, the formula $\psi(N)=\big( \frac{N^{2-\alpha}}{\phi(N)} \big)^2 \sim \frac{N^2}{\lambda^2}$ and Gauss counting argument \eqref{eq:gauss_count_beta_0_alt}, in this regime, the estimate \eqref{eq:finalestimate_change_var} grants us, as $N\to\infty$,
{\footnotesize
\begin{align*}
	& \mu_{N,N',N''}^+(f) - \nu_{N,N',N''}^+(f)
	\\ = & \bigO_{\alpha,\Lambda} \Big( \frac{A^2\lambda ^2\|f\|_\infty}{N}
	(1 + \lambda^{-2}) + \frac{A^3 \|f\|_\infty(\lambda^{-1}+1)}{N}
	+ \frac{ A^4 \|df\|_\infty \lambda^2}{\lambda^\frac{1}{1-\alpha}N} (\lambda^{-1}+1)^2 \sys_{\vec{\Lambda}}^{-\frac{1}{1-\alpha}}
	\Big)
	\\ = & \bigO_{\alpha,\Lambda} \Big( \frac{A^4 (\|f\|_\infty + \|df\|_\infty)(\lambda+\lambda^{-1})^{\max\{2,\frac{1}{1-\alpha}\}}}{N} \Big). \numberthis\label{eq:finalestimate_change_var_exotic_case}
\end{align*}
}Thanks to the estimates \eqref{eq:finalestimate_change_var_exotic_case}, \eqref{eq:estimate_nu_average_rat}, \eqref{eq:estimate_nu_average_irrat_2} and to Lemmas \ref{lem:negligible_diag_points} and \ref{lem:complex_linear_approx}, in this regime too we can focus on the behaviour of the sequence $(\nu_N)_{N\in\NN}$ defined in Equation \eqref{eq:def_nu_N}. Its density $g_N$ with respect to the Lebesgue measure of $\CC$ has the following the pointwise almost everywhere convergence outside of a countable union of circles: for all $z\in\CC- \bigcup_{p\in \vec{\Lambda}} C(0,\alpha\lambda |p|)$,
{\footnotesize
$$
g_N(z) = \frac{(\alpha \phi(N))^{\frac{2}{1-\alpha}} |z|^{-\frac{4-2\alpha}{1-\alpha}}}{(1-\alpha)\covol_{\vec{\Lambda}} \psi(N)} \sum_{\substack{p\in\vec{\Lambda}-\{0\} \\ |p| \leq \frac{|z| N^{1-\alpha}}{\alpha\phi(N)}}} |p|^\frac{2}{1-\alpha} \underset{N\to\infty}{\longrightarrow} \frac{\alpha^{\frac{2}{1-\alpha}}}{(1-\alpha) \covol_{\vec{\Lambda}}} \Big(\frac{|z|}{\lambda}\Big)^{-\frac{4-2\alpha}{1-\alpha}} \sum_{\substack{p\in\vec{\Lambda}-\{0\} \\ |p| \leq \frac{|z|}{\alpha \lambda}}} |p|^\frac{2}{1-\alpha},
$$
}which is the formula of the function $\rho_{\alpha,\vec{\Lambda},\lambda}$ defined before Theorem \ref{th:main}. In this section, we aim at making this convergence effective and at concluding the proof of Theorem \ref{th:effective_cv_complex_correlations}. From now on, we assume that $N$ is large enough so that $\frac{\lambda}{2} \leq \frac{\phi(N)}{N^{1-\alpha}} \leq 2\lambda$. First, one can notice that both functions $g_N$ and $\rho_{\alpha,\vec{\Lambda},\lambda}$ vanish on the open disk $D(0,\frac{\alpha \lambda \sys_{\vec{\Lambda}}}{2})$. For all $z \in D(0,A)$, we have the inequality
{\small
\begin{align*}
& |g_N(z)-\rho_{\alpha,\vec{\Lambda},\lambda}(z)| \leq \frac{\alpha^{\frac{2}{1-\alpha}} (\frac{\alpha \lambda \sys_{\vec{\Lambda}}}{2} )^{-\frac{4-2\alpha}{1-\alpha}}}{(1-\alpha) \covol_{\vec{\Lambda}}} \Big| \Big(\frac{\phi(N)}{N^{1-\alpha}}\Big)^\frac{4-2\alpha}{1-\alpha} - \lambda^\frac{4-2\alpha}{1-\alpha} \Big| \sum_{\substack{p\in\vec{\Lambda}-\{0\} \\ |p| \leq \frac{2A}{\alpha\lambda}}} |p|^\frac{2}{1-\alpha}
\\ + \; & \frac{\alpha^{\frac{2}{1-\alpha}} (\frac{\alpha \sys_{\vec{\Lambda}}}{2} )^{-\frac{4-2\alpha}{1-\alpha}}}{(1-\alpha) \covol_{\vec{\Lambda}}} \sum_{\substack{p\in\vec{\Lambda}-\{0\} \\ |p| \leq \frac{2A}{\alpha\lambda}}} |p|^\frac{2}{1-\alpha} \; \big|\11_{[\alpha \frac{\phi(N)}{N^{1-\alpha}} |p|, +\infty[\,}(|z|) - \11_{[\alpha \lambda |p|, +\infty[\,}(|z|) \big|.
\end{align*}
}Integrating on each annulus $\{z \in \CC \, : \, |z| \in [\alpha \lambda |p|, \alpha \frac{\phi(N)}{N^{1-\alpha}} |p|]\}$ and using Lemma \ref{lem:sum_powers_grid} (summing over $p'=\lambda p \in \lambda \vec{\Lambda}-\{0\}$ with $x=\frac{2A}{\alpha} \geq 1$ and $\beta=\frac{2}{1-\alpha}$ then $\beta=\frac{4-2\alpha}{1-\alpha}$), thanks to the inclusion $\supp{f} \subset D(0,A)$ and the inequality $\frac{\phi(N)}{N^{1-\alpha}} \leq 2\lambda$, we obtain the estimate, as $N\to\infty$,
{\footnotesize
\begin{align*}
	& \Big| \int_{\CC} (g_N-\rho_{\alpha,\vec{\Lambda},\lambda}) f \, d\Leb_\CC \Big| = \Big| \int_{D(0,A) - D(0,\frac{\alpha \lambda \sys_{\vec{\Lambda}}}{2})} (g_N-\rho_{\alpha,\vec{\Lambda},\lambda}) f \, d\Leb_\CC \Big|
	\\ = \, & A^2 \|f\|_\infty \bigO_\alpha \Big( \frac{(\lambda \sys_{\vec{\Lambda}})^{-\frac{4-2\alpha}{1-\alpha}}}{\covol_{\vec{\Lambda}}} \, \Big| \Big(\frac{\phi(N)}{N^{1-\alpha}} \Big)^\frac{4-2\alpha}{1-\alpha} - \lambda^\frac{4-2\alpha}{1-\alpha} \Big| \sum_{\substack{p\in\vec{\Lambda}-\{0\} \\ |p| \leq \frac{2A}{\alpha\lambda}}} |p|^\frac{2}{1-\alpha} \Big)
	\\ & + A^2 \|f\|_\infty
	\bigO_\alpha \Big( \frac{\sys_{\vec{\Lambda}}^{-\frac{4-2\alpha}{1-\alpha}} }{\covol_{\vec{\Lambda}}} \sum_{\substack{p\in\vec{\Lambda}-\{0\} \\ |p| \leq \frac{2A}{\alpha \lambda}}} |p|^\frac{2}{1-\alpha} 2\pi \max\Big\{ \alpha \frac{\phi(N)}{N^{1-\alpha}} |p|, \alpha\lambda |p| \Big\} \, \Big| \alpha \frac{\phi(N)}{N^{1-\alpha}} |p| - \alpha\lambda |p|\Big| \Big)
\end{align*}
\begin{align*}	
	\hspace{-1cm } = \, & \|f\|_\infty \bigO_{\alpha,\Lambda} \Big( \frac{A^\frac{6-4\alpha}{1-\alpha} (1+\lambda^2)}{\lambda^\frac{4-2\alpha}{1-\alpha}} \big|\big(\frac{\phi(N)}{\lambda N^{1-\alpha}}\big)^\frac{4-2\alpha}{1-\alpha}-1 \big| + \frac{A^\frac{8-6\alpha}{1-\alpha} (1+\lambda^2)}{\lambda^\frac{8-6\alpha}{1-\alpha}} \big| \frac{\phi(N)}{\lambda N^{1-\alpha}} - 1 \big|\Big)
	\\ \hspace{-1cm } = \, & \bigO_{\alpha,\Lambda} \Big(\|f\|_\infty A^\frac{8-6\alpha}{1-\alpha} (1+\lambda ^2) \big( \frac{1}{\lambda^{\frac{4-2\alpha}{1-\alpha}}} + \frac{1}{\lambda^{\frac{8-6\alpha}{1-\alpha}}} \big) \big| \frac{\phi(N)}{\lambda N^{1-\alpha}} - 1 \big| \Big).
\end{align*}
}Recalling that $\nu_N=g_N \Leb_\CC$, combining the latter estimate with the ones from Equations \eqref{eq:estimate_nu_average_rat}, \eqref{eq:estimate_nu_average_irrat_2}, \eqref{eq:finalestimate_change_var_exotic_case}, the symmetry described in Equation \eqref{eq:symmetry}, and Lemmas \ref{lem:negligible_diag_points} and \ref{lem:complex_linear_approx}, we have finally proven Theorem \ref{th:effective_cv_complex_correlations} (in which we simplified the error term by using standard inequalities such as $1+\lambda^\beta \leq 2(\lambda + \frac{1}{\lambda})^{|\beta|}$ for every real number $\beta$).
\qed

\section{Removing the branch cut}\label{sec:remove_branch_cut}
In the beginning of Section \ref{sec:lemmas}, we defined an empirical pair correlation measure $\R_{N,N',N''}^{\alpha,\Lambda}$. In its definition \eqref{eq:def_empirical_pair_cor_riemsurf}, for all grid points $n,m \in \Lambda$, we have the condition $|\Im(r)-\Im(s)|< 2\pi$ where $r,s$ are logarithms of $n,m$ in the associated Riemann surface $\widetilde{\CC^*}=\CC$. In terms of the levels introduced in Section \ref{ssec:levels}, this translates to consider all terms of the form $n^{[\alpha,k]} - m^{[\alpha,k]}$ (already taken into account in $\R_{N,N',N''}^{\alpha,\Lambda,\lvl}$), as well as the terms $n^{[\alpha,k+1]}-m^{[\alpha,k]}$ (resp.~$n^{[\alpha,k]}-m^{[\alpha,k+1]}$) for which the argument condition $\theta(n)<\theta(m)$ (resp.~$\theta(n)>\theta(m)$) holds. In other words, comparing the measure $\R_{N,N',N''}^{\alpha,\Lambda}$ with its level separated avatar $\R_{N,N',N''}^{\alpha,\Lambda,\lvl}$ defined in Equation \eqref{eq:def_correl_pair_lvl} and studied in Section \ref{sec:proof_thm}, we obtain
{\small
\begin{align*}
	\R_{N,N',N''}^{\alpha,\Lambda} - \R_{N,N',N''}^{\alpha,\Lambda,\lvl} = \frac{1}{(N'+N'')\psi(N)} \sum_{k=-N'}^{N''-2} & \sum_{\substack{n, m \in \Lambda, \, n\neq m \\ 0 < |n|,|m|\leq N}} \Delta_{\phi(N)(n^{[\alpha,k+1]}-m^{[\alpha,k]})}\11_{\theta(n)<\theta(m)}\numberthis\label{eq:diff_pair_correl_alpha_irrat}
	\\ & \hspace{1.4cm} + \Delta_{\phi(N)(n^{[\alpha,k]}-m^{[\alpha,k+1]})}\11_{\theta(n)>\theta(m)}.
\end{align*}
}

\begin{lemma}\label{lem:counting_remove_cut}
	Let $A>1$. For every integer $k\in\ZZ$, let
	{\footnotesize
	$$
	I_{N,A,k} = \big\{ (n,m) \in \Lambda^2 \; : \; n\neq m, 0<|n|,|m| \leq N, \, |n^{[\alpha,k+1]} - m^{[\alpha,k]}| \leq \frac{A}{\phi(N)} \mbox{ and } \theta(n)<\theta(m) \; \big\}.
	$$
	}Then we have, as $N\to\infty$, 
	$$ \card(I_{N,A,k}) = \bigO_{\alpha,\Lambda} \Big( N \big(\frac{ AN^{1-\alpha}}{\phi(N)}+1\big)^3 \, \11_{\lambda \neq +\infty} \Big).$$
\end{lemma}

\begin{proof}
	Let $(n,m) \in I_{N,A,k}$, set $\e=\frac{\theta(n)+\theta(m)}{2} \in \, ]0,2\pi[ \,$ and notice that
	\begin{align*}
		& \alpha (2\pi k+\theta(m)) \in \alpha 2\pi k + \alpha \, ]\e,2\pi] = \alpha 2\pi (k+1) + \alpha \, ]\e-2\pi, 0]
		\\ \mbox{ and } & \alpha(2\pi (k+1) + \theta(n)) \in \alpha 2\pi (k+1) + \alpha [0,\e[.
	\end{align*}
	Since in addition $\alpha \in \,]0,1[\,$ and the scaling factor $N\mapsto \phi(N)$ converges to $+\infty$, we claim that both points $m^{[\alpha,k]}$ and $n^{[\alpha,k+1]}$ are close to the ray $L_{\alpha,k}$ of argument $2\pi (k+1)\alpha$. Indeed, assume first that the segment $[n^{[\alpha,k+1]},m^{[\alpha,k]}]$ and the ray $L_{\alpha,k}$ don't intersect (which can happen only if $\alpha \geq \frac{1}{2})$. Applying Al-Kashi's law of cosines to the triangle with vertices $n^{[\alpha,k+1]}$, $m^{[\alpha,k]}$ and $0$ with angle $\omega \in [2\pi (1-\alpha), \pi]$ at $0$, we obtain the inequalities
	{\small
	\begin{align*}
		\Big(\frac{A}{\phi(N)}\Big)^2 & \geq |n^{[\alpha,k+1]}-m^{[\alpha,k]}|^2 = |n|^{2\alpha}+|m|^{2\alpha} - 2 |n|^\alpha |m|^\alpha \cos(\omega)
		\\ & \geq |n|^{2\alpha}+|m|^{2\alpha} - 2 |n|^\alpha |m|^\alpha \cos(2\pi (1-\alpha)) 
		\\ & = (|n|^\alpha-|m|^\alpha)^2 + 2|n|^\alpha |m|^\alpha(1-\cos(2\pi (1-\alpha))).
	\end{align*}
	}\begin{figure}[ht]
	\centering
	\scalebox{0.98}{
		\begin{adjustbox}{clip, trim=0.cm 0.cm 0.cm 0.cm, max width=\textwidth}
			\includegraphics{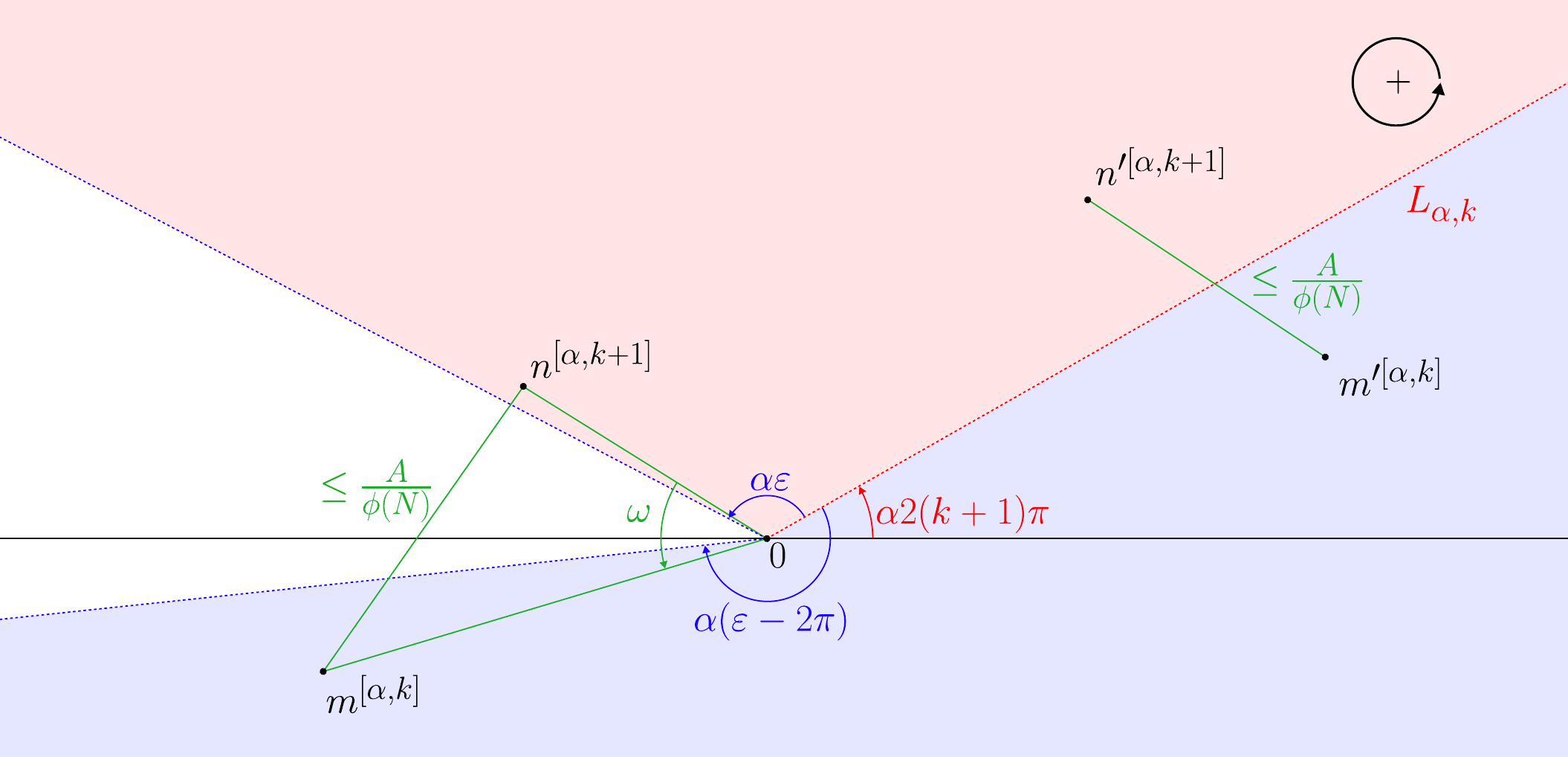}
		\end{adjustbox}
	}
	\caption{Illustration of the proof that both points $n^{[\alpha,k+1]}$ (confined in the \textcolor{red}{red} region) and $m^{[\alpha,k]}$ (in the \textcolor{blue}{blue} one) are close to the ray $L_{\alpha,k}$.}
	\label{fig:remove_cut_counting}
	\end{figure}
	
	\noindent Assuming that $|n| \leq |m|$ (instead of using $\min\{|n|,|m|\}$ and $\max\{|n|,|m|\}$), the latter equation and the triangle inequality on $|m|^\alpha=|m^{[\alpha,k]}|=|m^{[\alpha,k]}\pm n^{[\alpha,k+1]}|$ grant us the bound
	{\small
	\begin{equation}\label{eq:m^alpha_near_origin}
		|n|^\alpha \leq \frac{1}{\sqrt{2(1-\cos(2\pi (1-\alpha)))}} \frac{A}{\phi(N)} \mbox{ and } |m|^\alpha \leq \frac{A}{\phi(N)}+|n|^\alpha.
	\end{equation}
	}Then $n^{[\alpha,k+1]}$ and $m^{[\alpha,k]}$ are both close to $0$, hence close to $L_{\alpha,k}$.

	And now if $(n',m') \in I_{N,A,k}$ and if the segment $[n'^{[\alpha,k+1]},m'^{[\alpha,k]}]$ and the ray $L_{\alpha,k}$ intersect, we directly have the inequalities
	\begin{equation}\label{eq:m^alph_near_boundary}
		d(m'^{[\alpha,k]}, L_{\alpha,k}), \; d(n'^{[\alpha,k+1]}, L_{\alpha,k}) < \frac{A}{\phi(N)}.
	\end{equation}
	By Gauss counting argument \eqref{eq:gauss_count_beta_0_alt} and since $\frac{A}{\phi(N)} \to 0$, the inequalities \eqref{eq:m^alpha_near_origin} are only valid for a number $\bigO_\alpha(\frac{(1+\diam_{\vec{\Lambda}})^4}{\covol_{\vec{\Lambda}}^2})$ of indices $(n,m) \in \Lambda^2$. Hence, from now on we can assume that $[n^{[\alpha,k+1]}, m^{[\alpha,k]}]$ and the ray $L_{\alpha,k}$ do intersect and work with Equation \eqref{eq:m^alph_near_boundary}. Geometrically, this implies that (at least) one of the points $n^{[\alpha,k+1]}$ and $m^{[\alpha,k]}$ is in the open half-space centred at $L_{\alpha,k}$, i.e.~of equation $\Re(ze^{-i2\pi k\alpha})>0$. By symmetry, we can assume this holds for the point $m^{[\alpha,k]}$. Set
	$$
	P_\alpha : z \mapsto |z|^\alpha e^{i\alpha\omega} \mbox{ where } \arg(z) \equiv \omega \in 2\pi(k+1) + \,]\e-2\pi,\e[.
	$$
	This function coincides with $z \mapsto z^{[\alpha,k]}$ around $m$, and with $z\mapsto z^{[\alpha,k+1]}$ at $n$. Set $\e_N=\frac{A}{\phi(N)}$. Denote by $\ell$ a point in $L_{\alpha,k}$ for which the inequality $|m^{[\alpha,k]}-\ell|< \e_N$ holds. By the reversed triangle inequality, we see that $|\ell| \in D(|m|^\alpha, \e_N)$. Applying the mean value inequality to the inverse function of $P_\alpha$, we can locate the grid point $m$ as follows
	$$|m-P_\alpha^{-1}(\ell)| \leq |m^{[\alpha,k]} - \ell| \, \max_{D(\ell,\e_N)} |(P_\alpha^{-1})'| \leq \e_N \frac{1}{\alpha} (|\ell|+\e_N)^{\frac{1}{\alpha}-1}.$$
	In other words, the grid point $m$ is close to the positive real line in the following sense
	\begin{equation}\label{eq:m_close_negative_line}
		m \in D\Big(x_0, \, \frac{\e_N}{\alpha}(x_0^\alpha+\e_N)^{\frac{1}{\alpha}-1}\Big) \; \mbox{ where } \; x_0=P_\alpha^{-1}(\ell) \in \RR_+.
	\end{equation}
	We assume that $N$ is large enough so that the three inequalities, $\frac{2^{\frac{1}{\alpha}-1}\e_N}{\alpha} \leq \frac{1}{2}$, $\e_N \leq (\frac{\sys_{\Lambda}}{2})^\alpha$ and $\frac{\e_N}{\alpha}(\sys_{\Lambda}^\alpha+\e_N)^{\frac{1}{\alpha}-1} < \frac{\sys_{\Lambda}}{2}$ hold. The latter inequality implies that $x_0 \geq \frac{\sys_{\Lambda}}{2}$ for Equation $\eqref{eq:m_close_negative_line}$ to hold. One can then notice that the ball described in Equation \eqref{eq:m_close_negative_line} has radius bounded by $\frac{2^{\frac{1}{\alpha}-1}\e_N}{\alpha} x_0^{1-\alpha} \leq \frac{1}{2} x_0$. Since it is centred at $x_0 \geq \frac{\sys_{\Lambda}}{2}$, we obtain the inequality $\Re(m) \geq \frac{\sys_{\Lambda}}{4}$. More generally, for every real number $x \geq 0$, the ball from Equation \eqref{eq:m_close_negative_line} can intersect the vertical line above $x$ (or equivalently contains $x$) only if $x_0 \leq 2x$. Using the notation $L_g$ of Lemma \ref{lem:grid_pts_near_line}, this remark applied to $ x = \Re(m) \in [\frac{\sys_{\Lambda}}{4},|m|]$ implies that the point $m$ belongs to the set $L_g$ for the function
	$$ g:x \mapsto \frac{3^{\frac{1}{\alpha}-\alpha}\e_N}{\alpha} x^{1-\alpha}.$$
	Applying Lemma \ref{lem:grid_pts_near_line} to the grid $\Lambda$ and the inequalities $\max_{[k-1,k]}g=g(k)\leq g(N)$ (since $g$ is nondecreasing), this gives us the inequality
	\begin{equation}\label{eq:ineq_card_m_close_branch}
		\card \big\{ m \in \Lambda : \exists n \in \Lambda, (n,m) \in I_{N,A,k} \big\} \leq 4N \frac{(1+\diam_{\vec{\Lambda}})(g(N)+\diam_{\vec{\Lambda}})}{\covol_{\vec{\Lambda}}}.
	\end{equation}
	In order to count not only such points $m$ but all the ordered pairs $(n,m) \in I_{N,A,k}$, we will use the function $Q_\alpha : z \mapsto z^{[\frac{1}{\alpha},0]}$. The assumption $\theta(n)<\theta(m)$ gives the formula $\frac{n}{m}=Q_\alpha(\frac{n^{[\alpha,k+1]}}{m^{[\alpha,k]}})$. By applying the mean value inequality to the function $Q_\alpha$ between the points $1$ and $\frac{n^{[\alpha,k+1]}}{m^{[\alpha,k]}}$ on an adequate neighbourhood $V$ of $1$ (e.g.~ we can choose the half-disk $V=D(1,\frac{A}{\phi(N_0) \sys_{\Lambda}^\alpha}) \cap Q_\alpha(\CC^*)$ for $N_0$ large enough so that the closure of this half disk does not contain $0$), we obtain, with $c_\alpha = \max_V |Q_\alpha'|$,
	\begin{align*}
		& \big|1-\frac{n}{m}\big| \leq c_\alpha \big|1-\frac{n^{[\alpha,k+1]}}{m^{[\alpha,k]}}\big| \leq c_\alpha \frac{A}{\phi(N)|m|^\alpha}
		\\ \mbox{ i.e. } & |m-n| \leq c_\alpha \frac{A|m|^{1-\alpha}}{\phi(N)} \leq c_\alpha \frac{A N^{1-\alpha}}{\phi(N)}.\numberthis\label{eq:remove_cut_bound_m-n}
	\end{align*}
 	Using Gauss counting argument (more precisely, the right-hand inequality of Equation \eqref{eq:gauss_count_beta_0_alt} applied to the grid $m-\Lambda$), and recalling the definitions $\e_N=\frac{A}{\phi(N)}$ and $g(N) = \frac{3^{\frac{1}{\alpha}-\alpha}\e_N}{\alpha}N^{1-\alpha}$, the latter inequality yields, as $N\to\infty$,
	\begin{align*}
		\card(I_{N,A,k})
		& \leq 4N \frac{(1+\diam_{\vec{\Lambda}})(g(N)+\diam_{\vec{\Lambda}})}{\covol_{\vec{\Lambda}}} \frac{\pi}{\covol_{\vec{\Lambda}}} \big(c_\alpha \frac{AN^{1-\alpha}}{\phi(N)}+\diam_{\vec{\Lambda}}\big)^2
		\\ & = \bigO_{\alpha,\Lambda} \Big( N \big(\frac{AN^{1-\alpha}}{\phi(N)}+1\big)^3 \Big).
	\end{align*}
	In the case $\lambda \neq \infty$, this proves the lemma. If $\lambda = +\infty$, then Equation \eqref{eq:remove_cut_bound_m-n} becomes impossible as long as $N$ is large enough so that $\frac{\phi(N)}{N^{1-\alpha}} > c_\alpha \frac{A}{\sys_{\vec{\Lambda}}}$, hence $I_{N,A,k}$ is empty.
\end{proof}

\begin{proof}[Proof of Theorem \ref{th:main}]
	Immediate by combining Equation \eqref{eq:diff_pair_correl_alpha_irrat} with Lemma \ref{lem:counting_remove_cut}.
\end{proof}

\begin{remark}\label{rk:error_term_main}
	{\rm In addition, for all $f\in C^1_c(\CC)$ and $A>1$ such that $\supp{f} \subset D(0,A)$, we obtain the error term in Theorem \ref{th:main}, as $\min\{N,N'+N''\} \to \infty$,
	\begin{align*}
	\R_{N,N',N''}^{\alpha,\Lambda}(f) = & \int_\CC f(z) \rho_{\alpha,\vec{\Lambda},\lambda}(z) \, dz + \Err_{{Th.\ref{th:effective_cv_complex_correlations}}}(\alpha,\Lambda,f,A)
	\\ & + \bigO_{\alpha,\Lambda} \Big( \frac{N}{(N'+N'') \psi(N)} \big(\frac{ AN^{1-\alpha}}{\phi(N)}+1\big)^3 \, \11_{\lambda \neq +\infty} \Big).
	\end{align*}
	}
\end{remark}

\begin{proof}[Proof of Theorem \ref{th:main_rat}]
	Immediate by combining the rational version of Theorem \ref{th:effective_cv_complex_correlations} stated in Remark \ref{rk:effective_main_rat} with Lemma \ref{lem:counting_remove_cut}. 
\end{proof}

\begin{proof}[Proof of Theorem \ref{th:intro}]
	Let $\gamma \in \, ]0,1[\,$. Assuming $\phi(N)=N^\gamma$ and $\psi(N)=N^{2(2-\alpha-\gamma)}$, we can compare $\R_{N,0,b}^{\frac{1}{b},\Lambda}$ defined in Equation \eqref{eq:def_empirical_pair_cor_riemsurf} to the measure $\R_N$ defined in the introduction and obtain
	$$
	\R_N - \R_{N,0,b}^{\frac{1}{b},\Lambda} = \frac{1}{b\psi(N)} \, \sum_{\substack{m,n \in \Lambda, \, n \neq m \\ 0 < |m|,|n| \leq N}}  \; \sum_{\substack{ r \in \exp^{-1}(m), \, s \in \exp^{-1}(n) \\ |\Im(r)-\Im(s)| \geq 2 \pi \\ 0 \leq \Im(r), \, \Im(s) < 2\pi b}} \Delta_{\phi(N)(\exp(\frac{s}{b}) - \exp(\frac{r}{b}))}.
	$$
	Since $z \mapsto \exp(\frac{z}{b})$ induces a biholomorphism from $\CC/i 2\pi b \ZZ$ to $\CC^*$, for two points $\exp(\frac{r}{b})$ and $\exp(\frac{s}{b})$ to be close together, the associated classes $[r]$ and $[s]$ have to be close together in $\CC/i 2\pi b\ZZ$. Under the assumptions $|\Im(r)-\Im(s)| \geq 2 \pi$ and $0 \leq \Im(r), \Im(s) < 2\pi b$, this implies that one of the two points $r$,$s$ is close to the real line, and the other one to the horizontal line $\RR+i 2\pi b$. We use the notation $I_{N,A,b}$ from Lemma \ref{lem:counting_remove_cut}. Let $f \in C_c^1(\CC)$ and $A>1$ be such that $\supp{f} \subset D(0,A)$. For $N$ large enough, for an index $(n,m)$ to contribute to the sum  $\R_N(f) - \R_{N,0,b}^{\frac{1}{b},\Lambda}(f)$, then either $(n,m)$ or $(m,n)$ has to belong to $I_{N,A,b}$. The number of such points $n,m$ is evaluated in Lemma \ref{lem:counting_remove_cut}. Combining this with Theorem \ref{th:main_rat}, we obtain the vague convergence $\R_N - \R_{N,0,b}^{\frac{1}{b},\Lambda} \weakstar 0$ as $N\to +\infty$ and deduce Theorem \ref{th:intro}.
\end{proof}

\medskip
\noindent{\small {\it Fundings:} This work was supported by the French-Finnish CNRS IEA PaCAP.

\medskip
\noindent{\small {\it Acknowledgments:} The author would like to thank J.~Parkkonen and F.~Paulin, the supervisors of his ongoing doctorate, for their support, suggestions and corrections during this research. He also thanks the referee for their valuable inputs on an earlier version.}

\AtNextBibliography{\small}
\printbibliography[heading=bibintoc, title={References}]

\bigskip
{\small
	\noindent
	\begin{tabular}{l}
		Department of Mathematics and Statistics, P.O.~Box 35,\\
		40014 University of Jyv\"askyl\"a, FINLAND.\\
		{\it e-mail: sayousr@jyu.fi}
	\end{tabular}
}

\smallskip
and
\smallskip

{\small
	\noindent
	\begin{tabular}{l}
		Laboratoire de Mathématiques d'Orsay, UMR 8628 CNRS,\\
		Universit\'e Paris-Saclay, 91405 ORSAY Cedex, FRANCE.\\
		{\it e-mail: rafael.sayous@universite-paris-saclay.fr}
	\end{tabular}
}

\end{document}